\documentclass[twoside,a4paper,12pt]{article}

\usepackage{amsmath}
\usepackage{amsthm}
\usepackage{amsfonts}
\usepackage{hyperref}
\usepackage[a4paper, left=15mm, right=15mm]{geometry}

\newcommand{\g}{\ensuremath{\mathfrak g}}

\newcommand{\gl}{\ensuremath{\mathfrak g}{\mathfrak l}}
\newcommand{\go}{\ensuremath{{\mathfrak g}_0}}
\newcommand{\h}{\ensuremath{\mathfrak h}}
\newcommand{\kk}{\ensuremath{\mathfrak k}}
\newcommand{\kko}{\ensuremath{{\mathfrak k}_0}}

\newcommand{\po}{\ensuremath{{\mathfrak p}_0}}
\newcommand{\s}{\ensuremath{\mathfrak s}}

\newcommand{\su}{\ensuremath{{\mathfrak s}{\mathfrak u}}}
\newcommand{\esl}{\ensuremath{{\mathfrak s}{\mathfrak l}}}
\newcommand{\uu}{\ensuremath{\mathfrak u}}

\newcommand{\C}{\ensuremath{\mathbb C}}
\newcommand{\R}{\ensuremath{\mathbb R}}
\newcommand{\N}{\ensuremath{\mathbb N}}
\newcommand{\Z}{\ensuremath{\mathbb Z}}

\theoremstyle{definition}
\newtheorem{definition}{Definition}
\newtheorem{theorem}{Theorem}
\newtheorem{remark}{Remark}
\newtheorem{example}{Example}
\newtheorem{proposition}{Proposition}
\newtheorem{lemma}{Lemma}
\newtheorem{corollary}{Corollary}
\newtheorem{conjecture}{Conjecture}

\newcommand{\ds}{\displaystyle}

\begin{document}

\begin{center}
  {\bf  \Large A coefficient calculus for unitary $(\g,K)$ modules of $SU(2,2)$}
\end{center}
\vspace{5mm}

\begin{center}
  D. Kova\v{c}evi\'{c}{\footnote{e-mail:domagoj.kovacevic@fer.unizg.hr}}
  {\footnote{This work was supported by the QuantiXLie Centre of Excellence, a
  	project co-financed by the Croatian Government and European Union through
  	the European Regional Development Fund - the Competitiveness and Cohesion
  	Operational Programme (grant PK.1.1.02.0004).}}\\
  University of Zagreb, Faculty of Electrical Engineering and Computing,\\
  Unska 3, HR-10000 Zagreb, Croatia
\end{center}
\setcounter{page}{1}
\vspace{7mm}

\begin{abstract}
We develop an explicit operator and coefficient calculus for admissible
$(\g,K)$-modules of $SU(2,2)$, where
$\g=\mathfrak{sl}(4,\C)$ and $K=S(U(2)\times U(2))$. The $K$-types are indexed
by triples $(n,k,m)$, corresponding to the two $SU(2)$-factors and the central
$U(1)$-factor of $K$. We introduce explicit operators $A_\delta$ and
$B_\delta$, associated with the noncompact roots $\delta$, which act on highest
weight vectors of $K$-types, together with auxiliary operators $Q$ and $R$.
We prove their commutator relations, the four central coefficient relations,
and, in the unitary setting, all eight adjoint and norm-factor identities.
This reduces part of the analysis of $(\mathfrak g,K)$-modules to the study of
scalar coefficients obtained from compositions such as $A_\delta B_\delta$ and
$B_\delta A_\delta$.

Our coefficient formula applies only to weights of multiplicity one, and the
weights lying on the boundary of the region all have multiplicity one.
We assume that unitary $(\g,K)$ modules can be constructed from a certain set
of coefficients and we will analyze this construction in a subsequent paper.
We derive necessary sign restrictions together with explicit
relations that constrain the coefficient space. Moreover, these coefficients
also determine irreducibility, which we will address in a subsequent paper.
 
We analyze several families of unitary $(\g,K)$ modules
according to the minimal value of $n+m$. In the case $N=0$, the construction
leads to two-parameter families of unitary modules and describes how
reducibility occurs when certain coefficients vanish. For $N>0$, the method
produces both larger families of $K$-types and smaller multiplicity-free
families, including ladder-type modules. The resulting patterns are compatible
with the known description of the unitary dual of $SU(2,2)$ due to Knapp and
Speh. The paper provides an explicit computational framework for recovering
unitary $(\g,K)$-modules from their $K$-type structure and suggests a possible
approach to similar questions for other real reductive groups.
We emphasize once again that certain questions remain open, and we will address
them in future work.
\end{abstract}

\section{Introduction}

One of main goals of representation theory is determination of the unitary dual
of real reductive groups. The unitary dual, $\hat{G}$, of the group $G$, is the
set of equivalence classes of irreducible unitary representations of $G$. It
plays a crucial role in harmonic analysis on $G$ since it naturally appears in
the Plancherel theorem.

The main problem in the parametrization is the unitarity. Using the Langlands
classification, it is possible to parametrize the admissible dual
(the larger set than $\hat{G}$ which contains some nonunitary representations).
However, the problem of finding unitary representations in the admissible
dual is more subtler. Namely, one can start with nonunitary representation of
parabolic subgroup, induce it to representation of $G$ and obtain unitary
subquotient.

For the group $SU(2,2)$, this problem is particularly interesting since
the unitary dual exhibits many of the phenomena occurring in the general
theory of real reductive groups.
The group $SU(2,2)$ has real rank two and admits discrete series
representations, as well as several families obtained by parabolic
induction.
Its unitary dual contains discrete series and their limits, tempered
representations, complementary series, and certain isolated or
degenerate unitary representations.
The continuous families may involve one or two real parameters, while
the discrete series are parametrized by integral Harish--Chandra data.
The tempered part of the dual is described by parabolic induction from
discrete series representations of appropriate Levi subgroups, whereas
the non-tempered part requires a more delicate analysis of unitarity.
An explicit description of the irreducible unitary representations of
$SU(2,2)$ was obtained by Knapp and Speh \cite{KnappSpeh1982}, with the
Langlands classification providing a natural framework for organizing
the different families.

In order to have algebraic (and combinatorial) problem, we work with $(\g,K)$
modules. Here, \g\ is the complexified Lie algebra of $G$ and $K$ is the maximal
compact subgroup of $G$. Harish-Chandra realized that the essential structure
of the unitary representation $\pi$ lives in the set (space) of $K$-finite
vectors of $\pi$, denoted by $\pi_K$. Hence, it is enough to find $(\g,K)$
modules which admit positive-definite Hermitian form.

In our approach, we pick a highest weight vector for each $K$ type and then
modify the action of operators $X_\delta$ and $Y_\delta$ for noncompact roots
$\delta$. We obtain operators $A_\delta$ and $B_\delta$ acting only on the
highest weight vectors. At first sight, choosing an appropriate basis appears to
be problematic. But, we prefer to work with the compositions $A_\delta B_\delta$
and $B_\delta A_\delta$ when the space of $K$-types for some weight is
one-dimensional. In that case the compositions reduce to scalar coefficients.
We assume that the $(\g,K)$ module is completely determined by this set of
coefficients. We intend to prove this assumption in a subsequent work.
It turns out that these coefficients must be negative real numbers. They can be
parametrized by two real parameters leading to the sets of $K$-types obtained in
\cite{KnappSpeh1982}. 
We have developed a similar approach for the group
$SU(2,1)$ in \cite{kov2021}. Problem of finding unitary dual of the group
$SU(2,1)$ is much simpler since multiplicities of all weights are 1.

The highest weights of $K$-types are indexed by triples $(n,k,m)$ where
$n,m\in\mathbb Z_{\geq0}$ are the two compact highest weights and $k\in\Z$
is the central weight. The $K$-support is contained in a region bounded by
certain affine hyperplanes. In accordance with the standard description of the
$K$-types of irreducible admissible representations, and in particular with
Harish-Chandra's multiplicity-one result for the minimal $K$-type together with
the corresponding $K$-type multiplicity formulas, the $K$-types occurring on the
relevant boundary faces have multiplicity one, whereas higher multiplicities may
occur as one moves into the interior. In the computations below we therefore
begin with the action of the operators $A_\delta B_\delta$ and
$B_\delta A_\delta$ on these boundary $K$-types, where their action on the
highest-weight multiplicity space is scalar.

Several other questions concerning our construction remain open. We are
currently working on them and hope to resolve them positively in the near
future. However, we expect that answer to these questions will produce
2 or 3 papers of this size.

There is a natural question: why do we look for the unitary dual in this way?
We hope that we can develop a powerful technique which can be used to
analyze the unitary dual of other real reductive groups. Since the unitary
dual of $SU(2,2)$ is already known, we can compare our results. At the same
time, the group $SU(2,2)$ is the first group in which the multiplicities of
$K$-types are greater than 1. We hope that most of unpleasant technical
problems will be solved already for the group $SU(2,2)$.

We assume that our modules are irreducible in the generic case. For certain
special values of the parameters, the module may admit a proper submodule or a
nontrivial quotient.

Finally, the set of $K$-types appears to be very simple, which naturally
suggests describing the representation solely in terms of this set. The
corresponding coefficients then determine the set of $K$-types and explain
reducibility in a very natural way. We hope that the same will be possible for
other groups as well.

\section{Terminology}\label{term}

In this paper, we will mostly follow notation from \cite{kn}. Let 
\[ J = \begin{pmatrix}
	I_2 & 0 \\
	0 & -I_2\end{pmatrix} \in M_{4}(\mathbb{C}), \]
where $I_2$ denotes the $2 \times 2$ identity matrix. The group $SU(2,2)$ is
defined as
\[ SU(2,2) = \left\{g \in SL(4,\mathbb{C}) \;\middle|\; g^* J g = J\right\}, \]
where $g^* = \overline{g}^{\,t}$ is the conjugate transpose of $g$.

Thus $SU(2,2)$ consists of complex $4 \times 4$ matrices of determinant $1$
preserving a Hermitian form of signature $(2,2)$.

The Lie algebra $\mathfrak{su}(2,2)$ is given by
\[ \go=\mathfrak{su}(2,2)=\left\{X \in \mathfrak{sl}(4,\mathbb{C})\;\middle|\;
	X^* J + J X = 0\right\}.\]

Equivalently,
\[ \go=\su(2,2)=\left\{X \in M_4(\mathbb{C})\;\middle|\;
	\operatorname{tr}(X) = 0,\;X^* J + J X = 0\right\}. \]
The complexification of $\mathfrak{su}(2,2)$ is
\[ \g=(\g_0)_\C = \mathfrak{sl}(4,\C). \]
Thus $\mathfrak{su}(2,2)$ is a real form of $\mathfrak{sl}(4,\mathbb{C})$.
Let
\begin{equation*}
	\g=\h\oplus\left(\g_\alpha\oplus\g_\beta\oplus\g_\gamma\oplus
		\g_{\alpha+\beta}\oplus\g_{\beta+\gamma}\oplus
		\g_{\alpha+\beta+\gamma}\right)\oplus
		\left(\g_{-\alpha}\oplus\g_{-\beta}\oplus\g_{-\gamma}\oplus
		\g_{-\alpha-\beta}\oplus\g_{-\beta-\gamma}\oplus
		\g_{-\alpha-\beta-\gamma}\right)
\end{equation*}
be the root space decomposition. Hence, our simple roots are denoted by
$\alpha$, $\beta$ and $\gamma$ and, for example,
\begin{equation}\label{rootvec}
	H_\beta=\begin{pmatrix}0&0&0&0\\0&1&0&0\\0&0&-1&0\\0&0&0&0\end{pmatrix},
	X_\alpha=\begin{pmatrix}0&1&0&0\\0&0&0&0\\0&0&0&0\\0&0&0&0
		\end{pmatrix},
	Y_{\beta+\gamma}=\begin{pmatrix}0&0&0&0\\0&0&0&0\\0&0&0&0\\0&1&0&0
		\end{pmatrix}
\end{equation}

Define the Cartan involution $\theta$ on $\su(2,2)$ by
\[ \theta(X) = -X^*. \]
This yields the real Cartan decomposition
\[ \su(2,2) = \kko \oplus \po, \]
where
\[ \kko=\{X\in\su(2,2)\mid\theta(X)=X\},\quad
	\po=\{X\in\su(2,2)\mid\theta(X)=-X\}. \]

A maximal compact subgroup of $SU(2,2)$ is
\[ K = S(U(2) \times U(2)) = \left\{
	\begin{pmatrix}
		k_1 & 0 \\
		0 & k_2
	\end{pmatrix}\;\middle|\;k_1, k_2 \in U(2),\;\det(k_1)\det(k_2)=1\right\}.\]
However, it will be more convenient to think about $K$ as
\begin{equation}\label{K}
	K\cong(SU(2)\times SU(2)\times S^1)/\Z_2.
\end{equation}

The Lie algebra \kko, of $K$, is
\[ \kko=\s(\uu(2)\oplus\uu(2))=\left\{
	\begin{pmatrix}
		A & 0 \\
		0 & D
	\end{pmatrix}\;\middle|\;A, D \in \mathfrak{u}(2),\; \operatorname{tr}(A) +
	\operatorname{tr}(D) = 0\right\}. \]
Again, it will be more convenient to think about \kko\ as
\[ \kko=\su(2)\oplus\su(2)\oplus\uu(1)=\su(2)\oplus\su(2)\oplus\R \]
where $\uu(1)$ has the form
\[ \uu(1)=\left\{
	\begin{pmatrix}
		xiI_2 & 0 \\
		0 & -xiI_2
	\end{pmatrix}\;\middle|\;x\in\R\right\}\cong\R \]
and it commutes with both $\su(2)$.

The complexified Lie algebra is
\[ \kk=(\kko)_C \cong \mathfrak{s}(\gl(2,\mathbb{C})
	\oplus \mathfrak{gl}(2,\mathbb{C})). \]
Again, it will be more convenient to think about \kk\ as
\begin{equation}\label{k}
	\kk=\esl(2,\C)\oplus\esl(2,\C)\oplus\C
\end{equation}
where \C\ has the form
\[ \C=\left\{
\begin{pmatrix}
	xI_2 & 0 \\
	0 & -xI_2
\end{pmatrix}\;\middle|\;x\in\C\right\} \]
and it commutes with both $\esl(2,\C)$. It is easy to see the connection
between formulas \eqref{K} and \eqref{k}. Let us do one more step. In formula
\eqref{k}, the first \esl(2,\C) corresponds to the root $\alpha$ of
\g=\esl(4,\C). The second \esl(2,\C) corresponds to the root $\gamma$ of
\g=\esl(4,\C). The last piece, which we denoted by \C\ actually has the form
\begin{equation*}
	\C=\C(H_\alpha+2H_\beta+H_\gamma)=\C H_0.
\end{equation*}
Namely, $H_0$ commutes with $H_\alpha$, $X_\alpha$, $Y_\alpha$, $H_\gamma$,
$X_\gamma$ and $Y_\gamma$.

The subspace $\mathfrak{p}$ consists of matrices of the form
\[ \mathfrak{p}=\left\{
	\begin{pmatrix}
		0 & Z \\
		T & 0
	\end{pmatrix}\;\middle|\;Z,T \in M_2(\mathbb{C})\right\}. \]

The real rank of $SU(2,2)$ is $2$. A maximal abelian subspace 
$\mathfrak{a}_0\subset \mathfrak{p}_0$ can be chosen as
\[ \mathfrak{a}_0=\left\{
	\begin{pmatrix}
		0 & H \\
		H & 0
	\end{pmatrix}\;\middle|\;H = \operatorname{diag}(t_1, t_2),\; t_i 
	\in \mathbb{R}\right\}. \]
With respect to $\mathfrak{a}$, the restricted root system is of type $C_2$.

Let us analyze our Cartan subalgebra and root vectors. It is easy to see that
\begin{equation*}
	T_\alpha=\begin{pmatrix}0&0&0&0\\0&-i&0&0\\0&0&i&0\\0&0&0&0\end{pmatrix}
\end{equation*}
is an element of \go. Hence, by \eqref{rootvec},
\begin{equation}\label{xyagp}
	H_\alpha=iT_\alpha,\;H_\beta=iT_\beta,\;H_\gamma=iT_\gamma,\quad
	T_\alpha,T_\beta,T_\gamma\in\go.
\end{equation}
There is a difference between compact and noncompact roots. Roots $\alpha$ and
$\gamma$ are compact and
\begin{align}
	X_\alpha=\frac12\left(P_\alpha-iQ_\alpha\right),
	Y_\alpha=-\frac12\left(P_\alpha+iQ_\alpha\right),\\
	X_\gamma=\frac12\left(P_\gamma-iQ_\gamma\right),
	Y_\gamma=-\frac12\left(P_\gamma+iQ_\gamma\right),
\end{align}
for $P_\alpha,Q_\alpha,P_\gamma,Q_\gamma\in\go$. Let us mention that,
\begin{equation*}
	P_\alpha=\begin{pmatrix}0&1&0&0\\-1&0&0&0\\0&0&0&0\\0&0&0&0\end{pmatrix}
	\quad\mbox{and}\quad
	Q_\alpha=\begin{pmatrix}0&i&0&0\\i&0&0&0\\0&0&0&0\\0&0&0&0\end{pmatrix}.
\end{equation*}
Roots $\beta$, $\alpha+\beta$, $\beta+\gamma$ and $\alpha+\beta+\gamma$ are
noncompact and
\begin{gather}
	X_\beta=\frac12\left(P_\beta-iQ_\beta\right),
	Y_\beta=\frac12\left(P_\beta+iQ_\beta\right),\\
	X_{\alpha+\beta}=\frac12\left(P_{\alpha+\beta}-iQ_{\alpha+\beta}\right),
	Y_{\alpha+\beta}=\frac12\left(P_{\alpha+\beta}+iQ_{\alpha+\beta}\right),\\
	X_{\beta+\gamma}=\frac12\left(P_{\beta+\gamma}-iQ_{\beta+\gamma}\right),
	Y_{\beta+\gamma}=\frac12\left(P_{\beta+\gamma}+iQ_{\beta+\gamma}\right),\\
	X_{\alpha+\beta+\gamma}=
		\frac12\left(P_{\alpha+\beta+\gamma}-iQ_{\alpha+\beta+\gamma}\right),
	Y_{\alpha+\beta+\gamma}=\label{xyagz}
		\frac12\left(P_{\alpha+\beta+\gamma}+iQ_{\alpha+\beta+\gamma}\right)
\end{gather}
for $P_\delta,Q_\delta\in\go,\;\delta\in\{\beta,\alpha+\beta,\beta+\gamma,
\alpha+\beta+\gamma\}$. Let us mention that,
\begin{equation*}
	P_\beta=\begin{pmatrix}0&0&0&0\\0&0&1&0\\0&1&0&0\\0&0&0&0\end{pmatrix}
	\quad\mbox{and}\quad
	Q_\beta=\begin{pmatrix}0&0&0&0\\0&0&i&0\\0&-i&0&0\\0&0&0&0\end{pmatrix}.
\end{equation*}
The only difference between compact and noncompact roots is the minus sign on
$Y$ for compact roots. However, this difference is very important.

The center of $SU(2,2)$ is
\[ Z(SU(2,2)) = \{ \pm I_4, \pm iI_4 \}. \]
The group $SU(2,2)$ is a connected semisimple real Lie group of real rank $2$,
with complexified Lie algebra $\mathfrak{sl}(4,\mathbb{C})$, maximal compact
subgroup $S(U(2)\times U(2))$, and rich representation theory including discrete
series, tempered representations, and complementary series.

\section{\texorpdfstring{$(\g,K)$ modules of $SU(2,2)$}{(g,K) modules of
		SU(2,2)}}\label{secgk}

\subsection{\texorpdfstring{Operators $A_*$, $B_*$, $Q$ and $R$}
	{Operators A, B, Q and R}}

We begin with admissible $(\g,K)$-modules.  Unitarity is imposed only in the
later section on adjoints and signs.

We want to treat \kk\ modules as "points". It is enough to take the highest
weight vector $v$ instead of the whole \kk\ module and analyze the action
of \g. The action of \kk\ is unchanged. However, instead of action of vectors
$X_{\alpha+\beta+\gamma}$, $X_{\alpha+\beta}$, $X_{\beta+\gamma}$, $X_\beta$,
$Y_{\alpha+\beta+\gamma}$, $Y_{\alpha+\beta}$, $Y_{\beta+\gamma}$ and $Y_\beta$,
we have to consider the action of operators
$A_{\alpha+\beta+\gamma},\ldots,B_\beta$.
Usually, we simply say "the highest weight vector’ and omit the phrase "of the
\kk‑module’ or "of $K$‑type"; this will be understood from the context.

Let us recall the form of $\esl(2,\C)$ modules.
Let $H=\begin{pmatrix}1&0\\0&-1\end{pmatrix}$,
$X=\begin{pmatrix}0&1\\0&0\end{pmatrix}$ and
$Y=\begin{pmatrix}0&0\\1&0\end{pmatrix}$ be the basis of $\esl(2,C)$.
Let $n\in\N$. Then $V_n$ is spanned by vectors $\{v_{-n},v_{-n+2},\ldots,v_n\}$
and the action is given by
\begin{equation}\label{sl2}
	H.v_j=jv_j,\;\;X.v_j=\frac{n-j}2v_{j+2},\;\;Y.v_j=\frac{n+j}2v_{j-2},\;\;
	j\in\{-n,\ldots,n\},
\end{equation}
with assumption $v_{-n-2}=v_{n+2}=0$.
It is clear that $\operatorname{dim}V_n=n+1$.

Our $K$ types will have the form $V_n\otimes V_m$ where the first \esl(2,\C)
acts on $V_n$ and the second \esl(2,\C) acts on $V_m$. Since
$H_0=H_\alpha+2H_\beta+H_\gamma$ is in the center, it acts as a multiplication
by a scalar, say $k\in\Z$,
\begin{equation*}
	H_0.z=kz,\quad\forall z\in V_n\otimes V_m.
\end{equation*}
Hence, all $K$ types are determined by the triple $(n,k,m)$ where
$n,m\in\N\cup\{0\}$ and $k\in\Z$. We write
\begin{equation*}
	V=\bigoplus_{n,m\in\N\cup\{0\},k\in\Z}V_{n,k,m}.
\end{equation*}
where $\dim V_{n,k,m}\geq1$.
One could write $(V_n\otimes V_m)\otimes M_{n,k,m}$ instead of $V_{n,k,m}$ where
$M_{n,k,m}=\operatorname{Hom}_K(V_n\otimes V_m,V)$. However, it will be more
convenient to work with $V_{n,k,m}$ and then choose specific basis elements.
Let $v_{p,k,q}\in V_{n,k,m}$ be any highest weight vector. Then
\begin{equation*}
	H_\alpha.v_{p,k,q}=pv_{p,k,q},\quad
	X_\alpha.v_{p,k,q}=\frac{n-p}2v_{p+2,k,q},\quad
	Y_\alpha.v_{p,k,q}=\frac{n+p}2v_{p-2,k,q},\quad
\end{equation*}
and
\begin{equation*}
	H_\gamma.v_{p,k,q}=qv_{p,k,q},\quad
	X_\gamma.v_{p,k,q}=\frac{m-q}2v_{p,k,q+2},\quad
	Y_\gamma.v_{p,k,q}=\frac{m+q}2v_{p,k,q-2},\quad
\end{equation*}

Now, let us consider the action of all other elements of \g.
Let us recall that the action satisfies
\begin{equation}\label{form}
	[X,Y].v=X.(Y.v)-Y.(X.v)=X.Y.v-Y.X.v,\quad X,Y\in\g,\;v\in V
\end{equation}
where the action of $X$ on $v$ is denoted by $X.v$.
Sometimes, we will write $v_{p,k,q}(n,m)$ in order to emphasize the dimension
of the $K$ modules which contains that vector and also distinguish that module
from some other modules.
Let us consider the action of $X_{\alpha+\beta+\gamma}$ and let us act on the
highest weight vector $v_{n,k,m}(n,m)$. Later, we will reconstruct the action
to all vectors $v_{p,k,q}(n,m)$ of $V_{n,k,m}$
In order to analyze the element $X_{\alpha+\beta+\gamma}v_{n,k,m}(n,m)$,
we will act on it by $H_\alpha$, $H_\gamma$ and $H_0$ (and use \eqref{form})
\[ H_\alpha.X_{\alpha+\beta+\gamma}.v_{n,k,m}(n,m)=
	X_{\alpha+\beta+\gamma}.H_\alpha.v_{n,k,m}(n,m)+X_{\alpha+\beta+\gamma}.
	v_{n,k,m}(n,m)= \]
\[ =(n+1)X_{\alpha+\beta+\gamma}.v_{n,k,m}(n,m) \]
\[ H_\gamma.X_{\alpha+\beta+\gamma}.v_{n,k,m}(n,m)=
	X_{\alpha+\beta+\gamma}.H_\gamma.v_{n,k,m}(n,m)+X_{\alpha+\beta+\gamma}.
	v_{n,k,m}(n,m)= \]
\[ =(m+1)X_{\alpha+\beta+\gamma}.v_{n,k,m}(n,m) \]
\[ H_0.X_{\alpha+\beta+\gamma}.v_{n,k,m}(n,m)=
	X_{\alpha+\beta+\gamma}.H_0.v_{n,k,m}(n,m)+2X_{\alpha+\beta+\gamma}.
	v_{n,k,m}(n,m)= \]
\[ =(k+2)X_{\alpha+\beta+\gamma}.v_{n,k,m}(n,m) \]
It shows that
\begin{equation}\label{abcrs}
	X_{\alpha+\beta+\gamma}.v_{n,k,m}(n,m)=
	\bigoplus_{r,s\geq 0}v_{n+1,k+2,m+1}(n+1+2r,m+1+2s).
\end{equation}
Since,
\[ X_\alpha.X_{\alpha+\beta+\gamma}.v_{n,k,m}(n,m)=
	X_{\alpha+\beta+\gamma}.X_\alpha.v_{n,k,m,}(n,m)=0, \]
the action of $X_\alpha$ on the right hand side of \eqref{abcrs}
is 0 and hence $r=0$ in \eqref{abcrs}. Similarly, since
\[ X_\gamma.X_{\alpha+\beta+\gamma}.v_{n,k,m}(n,m)=
	X_{\alpha+\beta+\gamma}.X_\gamma.v_{n,k,m}(n,m)=0, \]
$s=0$ in \eqref{abcrs}. Hence
\begin{equation}\label{abc}
	X_{\alpha+\beta+\gamma}.v_{n,k,m}(n,m)=v_{n+1,k+2,m+1}(n+1,m+1).
\end{equation}
Since, $v_{n,k,m}(n,m)$ and $v_{n+1,k+2,m+1}(n+1,m+1)$ are highest weight
vectors, we define
\begin{equation*}
	A_{\alpha+\beta+\gamma}.v_{n,k,m}(n,m)=v_{n+1,k+2,m+1}(n+1,m+1).
\end{equation*}
or
\begin{equation*}
	A_{\alpha+\beta+\gamma}=X_{\alpha+\beta+\gamma}.
\end{equation*}

Now, let us consider the action of $X_{\alpha+\beta}$. Again,
$X_{\alpha+\beta}.v_{n,k,m}(n,m)$ will be analyzed by the action of
$H_\alpha$, $H_\gamma$ and $H_0$,
\[ H_\alpha.X_{\alpha+\beta}.v_{n,k,m}(n,m)=
	X_{\alpha+\beta}.H_\alpha.v_{n,k,m}(n,m)+X_{\alpha+\beta}.
	v_{n,k,m}(n,m)= \]
\[ =(n+1)X_{\alpha+\beta}.v_{n,k,m}(n,m) \]
\[ H_\gamma.X_{\alpha+\beta}.v_{n,k,m}(n,m)=
	X_{\alpha+\beta}.H_\gamma.v_{n,k,m}(n,m)-X_{\alpha+\beta}.
	v_{n,k,m}(n,m)= \]
\[ =(m-1)X_{\alpha+\beta}.v_{n,k,m}(n,m) \]
\[ H_0.X_{\alpha+\beta}.v_{n,k,m}(n,m)=
	X_{\alpha+\beta}.H_0.v_{n,k,m}(n,m)+2X_{\alpha+\beta}.
	v_{n,k,m}(n,m)= \]
\[ =(k+2)X_{\alpha+\beta}.v_{n,k,m}(n,m) \]
It shows that
\begin{equation}\label{abrs}
	X_{\alpha+\beta}.v_{n,k,m}(n,m)=
	\bigoplus_{r,s\geq 0}v_{n+1,k+2,m-1}(n+1+2r,m-1+2s).
\end{equation}
Since,
\[ X_\alpha.X_{\alpha+\beta}.v_{n,k,m}(n,m)=
	X_{\alpha+\beta}.X_\alpha.v_{n,k,m}(n,m)=0, \]
$r=0$ in \eqref{abrs}. However, the same result is not valid for $X_\gamma$.
Namely, $[X_\gamma,X_{\alpha+\beta}]=-X_{\alpha+\beta+\gamma}$ and the action of
$X_{\gamma}$ on the left hand side of \eqref{abrs} by \eqref{abc} is equal to
\[ X_\gamma.X_{\alpha+\beta}.v_{n,k,m}(n,m)=
	-X_{\alpha+\beta+\gamma}.v_{n,k,m}(n,m)=
	-v_{n+1,k+2,m+1}(n+1,m+1). \]
The action of $X_{\gamma}$ on the right hand side of \eqref{abrs}, by
\eqref{sl2} (and we already know that $r=0$) is equal to
\begin{equation*}
	\bigoplus_{s\geq 0}X_{\gamma}.v_{n+1,k+2,m-1}(n+1,m-1+2s)=
	\bigoplus_{s\geq 0}v_{n+1,k+2,m+1}(n+1,m-1+2s).
\end{equation*}
If follows that $s=0$ (and produces 0), or $s=1$ (and produces
$v_{n+1,k+2,m+1}(n+1,m+1)$). One can apply the action of $X_\gamma$ once more
and obtain that $s\leq1$. Hence,
\begin{equation}\label{abt}
	X_{\alpha+\beta}.v_{n,k,m}(n,m)=
	v_{n+1,k+2,m-1}(n+1,m-1)+v_{n+1,k+2,m-1}(n+1,m+1)
\end{equation}
Our calculation also reveals the form of $v_{n+1,k+2,m-1}(n+1,m+1)$. The action
of $X_\gamma$ on the previous formula produces
\[ -X_{\alpha+\beta+\gamma}.v_{n,k,m}(n,m)=v_{n+1,k+2,m+1}(n+1,m+1) \]
and the action of $Y_\gamma$ on that formula produces
\[ -Y_\gamma.X_{\alpha+\beta+\gamma}.v_{n,k,m}(n,m)=
	(m+1)v_{n+1,k+2,m-1}(n+1,m+1). \]
Hence
\[ v_{n+1,k+2,m-1}(n+1,m+1)=
	-\frac1{m+1}Y_\gamma.X_{\alpha+\beta+\gamma}.v_{n,k,m}(n,m)= \]
\[ -\frac1{H_\gamma(v_{n,k,m}(n,m))+1}Y_\gamma.A_{\alpha+\beta+\gamma}
	v_{n,k,m}(n,m) \]
and \eqref{abt} transforms to
\[ X_{\alpha+\beta}.v_{n,k,m}(n,m)=v_{n+1,k+2,m-1}(n+1,m-1)- \]
\[ -\frac1{H_\gamma(v_{n,k,m}(n,m))+1}Y_\gamma.A_{\alpha+\beta+\gamma}
	v_{n,k,m}(n,m) \]
One should notice that in that formula all vectors are highest weight vectors.
Hence, the operator that preserves highest weight vectors have the form
\[ X_{\alpha+\beta}+Y_\gamma X_{\alpha+\beta+\gamma}\frac1{H_\gamma+1} \]
where the fraction $\frac1{H_\gamma+1}$ means the multiplication by the value
$\frac1{m+1}$. It is more convenient and more natural to write down that
formula if we multiply it by $m+1$. It produces a definition
\begin{equation*}
	A_{\alpha+\beta}=
	X_{\alpha+\beta}(H_\gamma+1)+Y_\gamma X_{\alpha+\beta+\gamma}
\end{equation*}
It acts on the highest value $v_{n,k,m}(n,m)$ and, now, the terminology is
unambiguous. The action of $H_\gamma+1$ on $v_{n,k,m}(n,m)$ is
$(m+1)v_{n,k,m}(n,m)$ and $X_{\alpha+\beta}$ acts on it.

We can proceed, since the calculations are the same. Hence, we omit the full
proof, as the details are straightforward.

\begin{definition}\label{def1}
	The following operators acting on the highest weight vectors produce
	again the highest weight vectors
	\begin{enumerate}
		\item $A_{\alpha+\beta+\gamma}=X_{\alpha+\beta+\gamma}$
		\item $A_{\alpha+\beta}=X_{\alpha+\beta}(H_\gamma+1)
			+Y_\gamma X_{\alpha+\beta+\gamma}$
		\item $A_{\beta+\gamma}=X_{\beta+\gamma}(H_\alpha+1)
			-Y_\alpha X_{\alpha+\beta+\gamma}$
		\item $A_{\beta}=X_{\beta}(H_\alpha+1)(H_\gamma+1)
			-Y_\alpha X_{\alpha+\beta}(H_\gamma+1)\\
			\hspace*{10mm}+Y_\gamma X_{\beta+\gamma}(H_\alpha+1)
			-Y_\alpha Y_\gamma X_{\alpha+\beta+\gamma}$
		\item $B_{\alpha+\beta+\gamma}=Y_{\alpha+\beta+\gamma}(H_\alpha+1)
			(H_\gamma+1)+Y_\alpha Y_{\beta+\gamma}(H_\gamma+1)\\
			\hspace*{18mm}-Y_\gamma Y_{\alpha+\beta}(H_\alpha+1)
			-Y_\alpha Y_\gamma Y_\beta$
		\item $B_{\alpha+\beta}=Y_{\alpha+\beta}(H_\alpha+1)+Y_\alpha Y_\beta$
		\item $B_{\beta+\gamma}=Y_{\beta+\gamma}(H_\gamma+1)-Y_\gamma Y_\beta$
		\item $B_\beta=Y_\beta$
	\end{enumerate}
\end{definition}

\begin{remark}
	Expression for $A_{\beta+\gamma}$ contains minus sign and the expression for
	$A_{\alpha+\beta}$ contains the plus sign. The reason is very simple.
	The minus sign comes from the formula
	$[X_\alpha,X_{\beta+\gamma}]=X_{\alpha+\beta+\gamma}$.
	The plus sign comes from the formula
	$[X_\gamma,X_{\alpha+\beta}]=-X_{\alpha+\beta+\gamma}$.
\end{remark}

\begin{remark}
	Formulas for $A$ and $B$ are not symmetric since we always
	calculate the highest weight vector. If we calculate the
	lowest weight vectors in formulas for $B$, then expressions would
	be more symmetric.
\end{remark}

\begin{remark}\label{remhigh}
	One can prove that these expressions acting on the highest weight vectors
	produce again highest weight vectors by calculating formulas $[A,X_\alpha]$,
	$[A,X_\gamma]$, $[A,X_\alpha]$ and $[A,X_\gamma]$. The result is an
	operator that gives 0 when acting on highest weight vectors. For example,
	\[ [A_\beta,X_\alpha]=-X_{\alpha+\beta}(H_\alpha+1)(H_\gamma+1)+
		2X_\beta X_\alpha(H_\gamma+1)+
		H_\alpha X_{\alpha+\beta}(H_\gamma+1)- \]
	\[ Y_\gamma X_{\alpha+\beta+\gamma}(H_\alpha+1)+
		2Y_\gamma X_{\beta+\gamma}X_\alpha+
		H_\alpha Y_\gamma X_{\alpha+\beta+\gamma}= \]
	\[ 2(X_\beta X_\alpha(H_\gamma+1)+Y_\gamma X_{\beta+\gamma})
		X_\alpha \]
	and it produces 0 acting on the highest weight vector since the action of
	$X_\alpha$ on the highest weight vector of K module produces 0.
\end{remark}

We want to work with operators $A_*$ and $B_*$. However, we have to be able to
come back to our operators $X_*$ and $Y_*$.

\begin{proposition}\label{xyizab}
	Operators $X_*$ and $Y_*$ can be obtained from operators $A_*$ and $B_*$ by
	the following list of formulas,
	\begin{enumerate}
		\item $X_{\alpha+\beta+\gamma}=A_{\alpha+\beta+\gamma}$
		\item $X_{\alpha+\beta}=A_{\alpha+\beta}\frac1{H_\gamma+1}
			-Y_\gamma A_{\alpha+\beta+\gamma}\frac1{H_\gamma+1}$
		\item $X_{\beta+\gamma}=A_{\beta+\gamma}\frac1{H_\alpha+1}
			+Y_\alpha A_{\alpha+\beta+\gamma}\frac1{H_\alpha+1}$
		\item $X_{\beta}=A_{\beta}\frac1{(H_\alpha+1)(H_\gamma+1)}
			+Y_\alpha A_{\alpha+\beta}\frac1{(H_\alpha+1)(H_\gamma+1)}\\
			\hspace*{11mm}-Y_\gamma A_{\beta+\gamma}\frac1{(H_\alpha+1)
			(H_\gamma+1)}-Y_\alpha Y_\gamma A_{\alpha+\beta+\gamma}
			\frac1{(H_\alpha+1)(H_\gamma+1)}$
		\item $Y_{\alpha+\beta+\gamma}=B_{\alpha+\beta+\gamma}
			\frac1{(H_\alpha+1)(H_\gamma+1)}
			-Y_\alpha B_{\beta+\gamma}\frac1{(H_\alpha+1)(H_\gamma+1)}\\
			\hspace*{17mm}+Y_\gamma B_{\alpha+\beta}
			\frac1{(H_\alpha+1)(H_\gamma+1)}
			-Y_\alpha Y_\gamma B_\beta\frac1{(H_\alpha+1)(H_\gamma+1)}$
		\item $Y_{\alpha+\beta}=B_{\alpha+\beta}\frac1{H_\alpha+1}
			-Y_\alpha B_\beta\frac1{H_\alpha+1}$
		\item $Y_{\beta+\gamma}=B_{\beta+\gamma}\frac1{H_\gamma+1}
			+Y_\gamma B_\beta\frac1{H_\gamma+1}$
		\item $Y_\beta=B_\beta$
	\end{enumerate}
\end{proposition}

\begin{proof}
	All formulas easily follow from the Definition \eqref{def1}.
\end{proof}

\begin{remark}
	The fraction $\frac1{H_\alpha+1}$ simply means multiplication by
	$\frac1{n+1}$ where $n$ satisfies $H_\alpha v=nv$ where $v$ is the highest
	weight vector of some $K$ module.
\end{remark}

\begin{remark}
	Operators $X_*$ and $Y_*$ are described by $A_*$, $B_*$, $Y_\alpha$ and
	$Y_\gamma$. One should notice that $\beta$ is a noncompact root.
\end{remark}

We need two more operators acting on highest weight vectors of $K$ modules.

\begin{definition}
	Let
	\begin{equation*}
		Q=X_{\alpha+\beta}X_{\beta+\gamma}-X_\beta X_{\alpha+\beta+\gamma}
	\end{equation*}
	and
	\begin{equation*}
		R=Y_{\alpha+\beta}Y_{\beta+\gamma}-Y_\beta Y_{\alpha+\beta+\gamma}
	\end{equation*}
	be operators on the set of highest weight vectors of $K$ modules.
\end{definition}

\subsection{\texorpdfstring{Properties of operators $A_*$, $B_*$, $Q$ and $R$}
	{Properties of operators A, B, Q and R}}

At first glance, it looks artificial. However, it comes very naturally.

\begin{proposition}\label{commab}
	Our operators satisfy following formulas,
	\begin{equation}\label{comma}
		[A_{\alpha+\beta+\gamma},A_\beta]=Q(H_\alpha+H_\gamma+2)
	\end{equation}
	\begin{equation*}
		[A_{\alpha+\beta},A_{\beta+\gamma}]=Q(H_\alpha-H_\gamma)
	\end{equation*}
	\begin{equation}\label{commb}
		[B_\beta,B_{\alpha+\beta+\gamma}]=R(H_\alpha+H_\gamma+2)
	\end{equation}
	\begin{equation*}
		[B_{\alpha+\beta},B_{\beta+\gamma}]=R(H_\gamma-H_\alpha)
	\end{equation*}
\end{proposition}

Before the proof, we need

\begin{corollary}\label{corcomm}
	Let $x$, $y$, $z$ and $t$ be some operators. Then
	\begin{equation}\label{comm}
		[xy,zt]=[x,z]yt+x[y,z]t+z[x,t]y+zx[y,t]
	\end{equation}
\end{corollary}

\begin{proof}
	Since $[x,zt]=[x,z]t+z[x,t]$ and $[xy,z]=[x,z]y+x[y,z]$ it follows
	\[ [xy,zt]=[xy,z]t+z[xy,t]=[x,z]yt+x[y,z]t+z[x,t]y+zx[y,t] \]
	for all operators $x$, $y$, $z$ and $t$.
\end{proof}

\begin{proof}
	We only calculate $[A_{\alpha+\beta+\gamma},A_\beta]$ and.
	$[A_{\alpha+\beta},A_{\beta+\gamma}]$. Hence,
	\[ [A_{\alpha+\beta+\gamma},A_\beta]=[X_{\alpha+\beta+\gamma}, \]
	\[ [X_{\beta}(H_\alpha+1)(H_\gamma+1)-Y_\alpha X_{\alpha+\beta}(H_\gamma+1)
		+Y_\gamma X_{\beta+\gamma}(H_\alpha+1)
		-Y_\alpha Y_\gamma X_{\alpha+\beta+\gamma}]= \]
	\[ =-X_{\beta}X_{\alpha+\beta+\gamma}(H_\gamma+1)
		-X_{\beta}(H_\alpha+1)X_{\alpha+\beta+\gamma}+ \]
	\[ +X_{\beta+\gamma}X_{\alpha+\beta}(H_\gamma+1)
		+Y_\alpha X_{\alpha+\beta}X_{\alpha+\beta+\gamma}+ \]
	\[ +X_{\alpha+\beta}X_{\beta+\gamma}(H_\alpha+1)
		-Y_\gamma X_{\beta+\gamma}X_{\alpha+\beta+\gamma}+ \]
	\[ +X_{\beta+\gamma}Y_\gamma X_{\alpha+\beta+\gamma}
		-Y_\alpha X_{\alpha+\beta}X_{\alpha+\beta+\gamma}= \]
	\[ -X_{\beta}X_{\alpha+\beta+\gamma}(H_\alpha+H_\gamma+3)
		+X_{\alpha+\beta}X_{\beta+\gamma}(H_\alpha+H_\gamma+2)
		-X_{\beta}X_{\alpha+\beta+\gamma}= \]
	\[ =Q(H_\alpha+H_\gamma+2) \]
	For the second formula we use the formula \eqref{comm},
	\[ [A_{\alpha+\beta},A_{\beta+\gamma}]= \]
	\[ [X_{\alpha+\beta}(H_\gamma+1)+Y_\gamma X_{\alpha+\beta+\gamma},
		X_{\beta+\gamma}(H_\alpha+1)-Y_\alpha X_{\alpha+\beta+\gamma}]= \]
	\[ =X_{\alpha+\beta}X_{\beta+\gamma}(H_\alpha+1)
		-X_{\beta+\gamma}X_{\alpha+\beta}(H_\gamma+1)- \]
	\[ -X_\beta X_{\alpha+\beta+\gamma}(H_\alpha+1)
		-X_{\beta+\gamma}Y_\gamma X_{\alpha+\beta+\gamma}+ \]
	\[ +X_\beta X_{\alpha+\beta+\gamma}(H_\gamma+1)
		-Y_\alpha X_{\alpha+\beta}X_{\alpha+\beta+\gamma}+ \]
	\[ +Y_\gamma X_{\beta+\gamma}X_{\alpha+\beta+\gamma}
		+Y_\alpha X_{\alpha+\beta}X_{\alpha+\beta+\gamma}= \]
	\[ =X_{\alpha+\beta}X_{\beta+\gamma}(H_\alpha-H_\gamma)
		-X_{\beta}X_{\alpha+\beta+\gamma}(H_\alpha-H_\gamma)=
		Q(H_\alpha-H_\gamma). \]
	The rest is the same.
\end{proof}

This proof indicates how the other identities can be obtained.
One should use the Definition \eqref{def1}, Formula \eqref{comm} and patiently
perform the calculation.

\begin{remark}
	We already saw in Remark \eqref{remhigh} that the certain operator takes
	highest weight vectors of $K$ modules to highest weight vectors if and only
	if it commutes with $X_\alpha$ and $X_\gamma$. Hence, let us calculate
	the commutator $[Q,X_\alpha]$,
	\[ [Q,X_\alpha]=[X_{\alpha+\beta}X_{\beta+\gamma}-
		X_\beta X_{\alpha+\beta+\gamma},X_\alpha]=
		-X_{\alpha+\beta}X_{\alpha+\beta+\gamma}
		+X_{\alpha+\beta}X_{\alpha+\beta+\gamma}=0. \]
	Similarly, $[Q,X_\gamma]=[R,X_\gamma]=[R,X_\gamma]=0$.
\end{remark}

Now, let us list formulas that our operators satisfy.

\begin{proposition}\label{commabqr}
	On $K$-highest vectors, operators $A_*$, $B_*$, $Q$ and $R$ satisfy
	\begin{enumerate}
		\item $[A_\beta,R]=B_{\alpha+\beta+\gamma}(1-H_\beta)$
		\item $[B_\beta,Q]=A_{\alpha+\beta+\gamma}(H_\beta-1)$
		\item $[B_{\alpha+\beta+\gamma},Q]=A_\beta(H_\alpha+H_\beta+H_\gamma+1)$
		\item $[A_{\alpha+\beta+\gamma},R]=
			-B_\beta(H_\alpha+H_\beta+H_\gamma+1)$
		\item $[B_{\alpha+\beta},Q]=-A_{\beta+\gamma}(H_\alpha+H_\beta)$
		\item $[B_{\beta+\gamma},Q]=-A_{\alpha+\beta}(H_\beta+H_\gamma)$
		\item $[A_{\alpha+\beta},R]=B_{\beta+\gamma}(H_\alpha+H_\beta)$
		\item $[A_{\beta+\gamma},R]=B_{\alpha+\beta}(H_\beta+H_\gamma)$
		\item $[A_\beta,Q]=0$
		\item $[A_{\alpha+\beta+\gamma},Q]=0$
		\item $[B_\beta,R]=0$
		\item $[B_{\alpha+\beta+\gamma},R]=0$
		\item $[A_{\alpha+\beta},Q]=0$
		\item $[A_{\beta+\gamma},Q]=0$
		\item $[B_{\alpha+\beta},R]=0$
		\item $[B_{\beta+\gamma},R]=0$
		\item $[A_{\alpha+\beta+\gamma},A_{\alpha+\beta}]=0$
		\item $[A_{\alpha+\beta+\gamma},A_{\beta+\gamma}]=0$
		\item $[A_\beta,A_{\alpha+\beta}]=0$
		\item $[A_\beta,A_{\beta+\gamma}]=0$
		\item $[B_{\alpha+\beta+\gamma},B_{\alpha+\beta}]=0$
		\item $[B_{\alpha+\beta+\gamma},B_{\beta+\gamma}]=0$
		\item $[B_\beta,B_{\alpha+\beta}]=0$
		\item $[B_\beta,B_{\beta+\gamma}]=0$
		\item $[A_{\alpha+\beta},B_{\beta+\gamma}]=0$
		\item $[A_{\beta+\gamma},B_{\alpha+\beta}]=0$
		\item $[A_\beta,B_{\alpha+\beta+\gamma}]=0$
	\end{enumerate}
\end{proposition}

\begin{proof}
	All relations follow by direct substitution of the definitions of the
	operators, using the standard commutation relations
	\[ [H,E_\rho]=\rho(H)E_\rho \]
	and the corresponding commutation relations between the root vectors.
	Relations 9--27 hold as identities in ($U(\mathfrak{sl}4)$. For relations
	1--8 we additionally use the fact that $v$ is a $K$-highest weight vector,
	so that
	\[ X_\alpha v=X_\gamma v=0, \]
	together with the prescribed Cartan eigenvalues. The computations are straightforward, although some of them are rather lengthy.
\end{proof}

\subsection{\texorpdfstring{Coefficients in $(\g,K)$ modules}
	{Coefficients in (g,K) modules}}\label{coef}

Analysis of $(\g,K)$ modules using our operators, naturally requires a basis
and this poses a problem. Construction of a basis of highest weight vectors can
be a very hard problem. Instead of this construction, we will analyze
composition of operators $A_\delta B_\delta$ and $B_\delta A_\delta$ for
$\delta\in\{\beta,\alpha+\beta,\beta+\gamma,\alpha+\beta+\gamma\}$.

Assume that a $K$-type occurs with multiplicity one, and let $v$ be its highest
weight vector. This assumption is crucial and is satisfied for weights on the
boundary of the set of highest weights. Then
\begin{equation}\label{coefdef}
	A_\delta B_\delta v=\lambda v
\end{equation}
and
\[ A_\delta B_\delta(\mu v)=\mu A_\delta B_\delta v=
	\mu\lambda v=\lambda(\mu v). \]
It shows that the composition uniquely defines $\lambda$ and it is
independent of the choice of the basis vector $v$.

\subsection{Parity classes}

\begin{proposition}\label{parity}
	Let
	\[ \Lambda=\{(n,k,m)\in\mathbb Z^3:n+k+m\equiv0\pmod2\}. \]
	The lattice generated by the formal $A$- and $B$-weight shifts has index
	four in $\Lambda$.  Its four cosets are
	\[ \Lambda_r=\{(n,k,m)\in\Lambda:2n+k\equiv r\pmod4\},\qquad r=0,1,2,3. \]
	Representatives are $(0,0,0)$, $(0,1,1)$, $(1,0,1)$ and $(1,1,0)$
	respectively.  This is a statement about the formal weight lattice; it does
	not assert that the corresponding operators are nonzero or act transitively
	on the set of $K$ types of a module.
\end{proposition}

\begin{proof} In $(n,k,m)$ order, the four $A$-shifts are
	\[ (1,2,1),\quad(1,2,-1),\quad(-1,2,1),\quad(-1,2,-1), \]
	and the $B$-shifts are their negatives.  Every shift preserves both
	$n+k+m$ modulo $2$ and $2n+k$ modulo $4$.  The determinant of the first
	three displayed shifts is $8$, so their lattice has index $8$ in
	$\mathbb Z^3$ and therefore index $4$ in the index-two lattice $\Lambda$.
	The four displayed representatives lie in $\Lambda$ and give residues
	$0,1,2,3$; hence the residue fibers are exactly the four cosets.
\end{proof}

\subsection{Commutators $[Q,R]$,
	$[B_{\alpha+\beta+\gamma},A_{\alpha+\beta+\gamma}]$ and $[B_\beta,A_\beta]$}

Let us calculate commutators $[Q,R]$,
$[B_{\alpha+\beta+\gamma},A_{\alpha+\beta+\gamma}]$ and $[B_\beta,A_\beta]$.

\begin{proposition}\label{qr}
	The commutator $[Q,R]$ is given by
	\begin{align*}
		[Q,R]&=2H_\beta(H_\alpha+H_\beta+H_\gamma+1)+H_\alpha H_\gamma+\\
			 &+Y_\beta X_\beta(H_\alpha+H_\beta+H_\gamma+2)
			 	+Y_{\alpha+\beta}X_{\alpha+\beta}(H_\beta+H_\gamma)+\\
			 &+Y_{\beta+\gamma}X_{\beta+\gamma}(H_\alpha+H_\beta)
			 	+Y_{\alpha+\beta+\gamma}X_{\alpha+\beta+\gamma}(H_\beta-2)-\\
			 &-Y_\alpha Y_\beta X_{\alpha+\beta}
			 	-Y_\alpha Y_{\beta+\gamma}X_{\alpha+\beta+\gamma}
			 	+Y_\gamma Y_\beta X_{\beta+\gamma}
			 	+Y_\gamma Y_{\alpha+\beta}X_{\alpha+\beta+\gamma}.
	\end{align*}
\end{proposition}

\begin{proof}
	The proof is a straightforward computation using Corollary~\ref{corcomm},
	\begin{equation*}
		[Q,R]=[X_{\alpha+\beta}X_{\beta+\gamma}-X_\beta X_{\alpha+\beta+\gamma},
			Y_{\alpha+\beta}Y_{\beta+\gamma}-Y_\beta Y_{\alpha+\beta+\gamma}]=
	\end{equation*}
	\begin{equation*}
		=(H_\alpha+H_\beta)X_{\beta+\gamma}Y_{\beta+\gamma}
		+Y_{\alpha+\beta}X_{\alpha+\beta}(H_\beta+H_\gamma)-
	\end{equation*}
	\begin{equation*}
		-Y_\alpha X_{\alpha+\beta+\gamma}Y_{\beta+\gamma}
		+X_\beta X_\gamma Y_{\beta+\gamma}
		+Y_{\alpha+\beta}Y_\gamma X_{\alpha+\beta+\gamma}
		-Y_{\alpha+\beta}X_\beta X_\alpha-
	\end{equation*}
	\begin{equation*}
		-X_\alpha X_{\beta+\gamma}Y_{\alpha+\beta+\gamma}
		+X_{\alpha+\beta}X_\gamma Y_{\alpha+\beta+\gamma}
		+Y_\beta Y_\gamma X_{\beta+\gamma}
		-Y_\beta X_{\alpha+\beta}Y_\alpha+
	\end{equation*}
	\begin{equation*}
		+H_\beta X_{\alpha+\beta+\gamma}Y_{\alpha+\beta+\gamma}
		+Y_\beta X_\beta(H_\alpha+H_\beta+H_\gamma)=
	\end{equation*}
	We want to write all terms in the form $Y_*X_*H_*$.
	One should notice that $H_\alpha X_{\beta+\gamma}Y_{\beta+\gamma}=
	X_{\beta+\gamma}Y_{\beta+\gamma}H_\alpha$ as well as several similar
	identities.
	Furthermore, the term $Y_{\alpha+\beta}X_\beta X_\alpha$ may be omitted,
	since the action on the highest weight vector is 0,
	\begin{equation*}
		=(Y_{\beta+\gamma}X_{\beta+\gamma}+H_\beta+H_\gamma)(H_\alpha+H_\beta)
		+Y_{\alpha+\beta}X_{\alpha+\beta}(H_\beta+H_\gamma)-
	\end{equation*}
	\begin{equation*}
		-Y_\alpha Y_{\beta+\gamma}X_{\alpha+\beta+\gamma}
		+(Y_\beta X_\beta+H_\beta)
		+(Y_\gamma Y_{\alpha+\beta}-Y_{\alpha+\beta+\gamma})
			X_{\alpha+\beta+\gamma}-
	\end{equation*}
	\begin{equation*}
		-(Y_{\alpha+\beta+\gamma}X_{\alpha+\beta+\gamma}
			-Y_{\beta+\gamma}X_{\beta+\gamma}+H_\alpha)
		+(Y_{\alpha+\beta}X_{\alpha+\beta}+H_\alpha+H_\beta)+
	\end{equation*}
	\begin{equation*}
		+(Y_\gamma Y_\beta X_{\beta+\gamma}-Y_{\beta+\gamma}X_{\beta+\gamma})
		-(-Y_\beta X_\beta+Y_\alpha Y_\beta X_{\alpha+\beta}
			+Y_{\alpha+\beta}X_{\alpha+\beta})+
	\end{equation*}
	\begin{equation*}
		+(Y_{\alpha+\beta+\gamma}X_{\alpha+\beta+\gamma}
			+H_\alpha+H_\beta+H_\gamma)H_\beta
		+Y_\beta X_\beta(H_\alpha+H_\beta+H_\gamma).
	\end{equation*}
	This completes the proof.
\end{proof}

\begin{proposition}\label{baabc}
	The commutator $[B_{\alpha+\beta+\gamma},A_{\alpha+\beta+\gamma}]$ is given
	by
	\begin{align*}
		[B_{\alpha+\beta+\gamma},A_{\alpha+\beta+\gamma}]&
		=-H_\alpha H_\gamma(H_\alpha+H_\beta+H_\gamma+2)-\\
		&-Y_\beta X_\beta+Y_{\alpha+\beta}X_{\alpha+\beta}(H_\alpha+1)+\\
		&+Y_{\beta+\gamma}X_{\beta+\gamma}(H_\gamma+1)
		+Y_{\alpha+\beta+\gamma}X_{\alpha+\beta+\gamma}(H_\alpha+H_\gamma+3)+\\
		&+Y_\alpha Y_\beta X_{\alpha+\beta}
		+Y_\alpha Y_{\beta+\gamma}X_{\alpha+\beta+\gamma}
		-Y_\gamma Y_\beta X_{\beta+\gamma}
		-Y_\gamma Y_{\alpha+\beta}X_{\alpha+\beta+\gamma}.
	\end{align*}
\end{proposition}

\begin{proof}
	The proof is similar to the proof of the previous proposition,
	\begin{equation*}
		[B_{\alpha+\beta+\gamma},A_{\alpha+\beta+\gamma}]=
	\end{equation*}
	\begin{equation*}
		=[Y_{\alpha+\beta+\gamma}(H_\alpha+1)(H_\gamma+1)
		+Y_\alpha Y_{\beta+\gamma}(H_\gamma+1)
		-Y_\gamma Y_{\alpha+\beta}(H_\alpha+1)-Y_\alpha Y_\gamma Y_\beta,
		X_{\alpha+\beta+\gamma}]=
	\end{equation*}
	\begin{equation*}
		=(H_\alpha+H_\beta+H_\gamma)(H_\alpha+1)(H_\gamma+1)
		+Y_{\alpha+\beta+\gamma}X_{\alpha+\beta+\gamma}(H_\gamma+1)
		+Y_{\alpha+\beta+\gamma}(H_\alpha+1)X_{\alpha+\beta+\gamma}+
	\end{equation*}
	\begin{equation*}
		+X_{\beta+\gamma}Y_{\beta+\gamma}(H_\gamma+1)
		+Y_\alpha Y_{\beta+\gamma}X_{\alpha+\beta+\gamma}+
	\end{equation*}
	\begin{equation*}
		+X_{\alpha+\beta}Y_{\alpha+\beta}(H_\alpha+1)
		-Y_\gamma Y_{\alpha+\beta}X_{\alpha+\beta+\gamma}-
	\end{equation*}
	\begin{equation*}
		-X_{\beta+\gamma}Y_\gamma Y_\beta
		+Y_\alpha X_{\alpha+\beta}Y_\beta=
	\end{equation*}
	\begin{equation*}
		=-(H_\alpha+H_\beta+H_\gamma)(H_\alpha+1)(H_\gamma+1)
		+Y_{\alpha+\beta+\gamma}X_{\alpha+\beta+\gamma}(H_\alpha+H_\gamma+3)+
	\end{equation*}
	\begin{equation*}
		+(Y_{\beta+\gamma}X_{\beta+\gamma}+H_\beta+H_\gamma)(H_\gamma+1)
		+Y_\alpha Y_{\beta+\gamma}X_{\alpha+\beta+\gamma}+
	\end{equation*}
	\begin{equation*}
		+(X_{\alpha+\beta}Y_{\alpha+\beta}+H_\alpha+H_\beta)(H_\alpha+1)
		-Y_\gamma Y_{\alpha+\beta}X_{\alpha+\beta+\gamma}-
	\end{equation*}
	\begin{equation*}
		-(Y_\gamma Y_\beta X_{\beta+\gamma}+Y_\beta X_\beta+H_\beta)
		+Y_\alpha Y_\beta X_{\alpha+\beta}=
	\end{equation*}
	This completes the proof.
\end{proof}

\begin{proposition}\label{bab}
	The commutator $[B_\beta,A_\beta]$ is given by
	\begin{align*}
		[B_\beta,A_\beta]&=-H_\beta(H_\alpha+2)(H_\gamma+2)-\\
		&-Y_\beta X_\beta(H_\alpha+H_\gamma+3)
		-Y_{\alpha+\beta}X_{\alpha+\beta}(H_\gamma+1)-\\
		&-Y_{\beta+\gamma}X_{\beta+\gamma}(H_\alpha+1)
		+Y_{\alpha+\beta+\gamma}X_{\alpha+\beta+\gamma}+\\
		&+Y_\alpha Y_\beta X_{\alpha+\beta}
		+Y_\alpha Y_{\beta+\gamma}X_{\alpha+\beta+\gamma}
		-Y_\gamma Y_\beta X_{\beta+\gamma}
		-Y_\gamma Y_{\alpha+\beta}X_{\alpha+\beta+\gamma}.
	\end{align*}
\end{proposition}

\begin{proof}
	The proof is analogous to the previous proof and is therefore omitted.
\end{proof}

At first sight, the formulas appear to be rather complicated.
However, they all contain the same expression:
$Y_\alpha Y_\beta X_{\alpha+\beta}
+Y_\alpha Y_{\beta+\gamma}X_{\alpha+\beta+\gamma}
-Y_\gamma Y_\beta X_{\beta+\gamma}
-Y_\gamma Y_{\alpha+\beta}X_{\alpha+\beta+\gamma}$.
The properties of this expression may be worthy of further investigation.

It turns out that the remaining terms in the commutators are also related. To
describe this relation, we first need a definition.

\begin{definition}
	The operator $C$ is defined by
	\begin{equation*}
		C=2H_\beta-H_\alpha H_\gamma+Y_\beta X_\beta
		+Y_{\alpha+\beta}X_{\alpha+\beta}+Y_{\beta+\gamma}X_{\beta+\gamma}
		+Y_{\alpha+\beta+\gamma}X_{\alpha+\beta+\gamma}.
	\end{equation*}
\end{definition}

\begin{remark}
	The operator $C$ possesses several interesting properties that merit further
	investigation. In particular, one can identify the Casimir operator in its
	expression, so that $C$ becomes a function of $H_\alpha$, $H_\beta$, and
	$H_\gamma$.
\end{remark}

\begin{proposition}
	The following identity holds:
	\begin{equation}\label{aqr}
		[B_{\alpha+\beta+\gamma},A_{\alpha+\beta+\gamma}]+[Q,R]=
		C(H_\alpha+H_\beta+H_\gamma+1).
	\end{equation}
\end{proposition}

\begin{proof}
	The result directly from Propositions \ref{baabc} and \ref{qr}.
\end{proof}

\begin{proposition}
	The following identity holds:
	\begin{equation}\label{bqr}
		[B_\beta,A_\beta]+[Q,R]=C(H_\beta-1).
	\end{equation}
\end{proposition}

\begin{proof}
	The result directly from Propositions \ref{bab} and \ref{qr}.
\end{proof}

Now, we want to eliminate the operator $C$ and obtain a relation
connecting the coefficients. This relation will be used later.

\begin{theorem}\label{qrel}
	The following identity holds:
	\begin{equation}\label{qrrel}
		(H_\alpha+H_\gamma+2)[Q,R]=
		(H_\beta-1)[B_{\alpha+\beta+\gamma},A_{\alpha+\beta+\gamma}]
		+(H_\alpha+H_\beta+H_\gamma+1)[A_\beta,B_\beta]
	\end{equation}
\end{theorem}

\begin{proof}
	It is enough to compute
	$(H_\beta-1)\cdot\eqref{aqr}-(H_\alpha+H_\beta+H_\gamma+1)\cdot$\eqref{bqr}.
\end{proof}

\section{\texorpdfstring{Unitary $(\g,K)$ modules}{Unitary (g,K) modules}}

\subsection{Adjoint operators}

Now, let us consider unitary $(\g,K)$ modules. The inner product will be denoted
by $(\cdot,\cdot)$ and the adjoint by $*$. We will no longer use $*$ to denote
the conjugate transpose. The $(\g,K)$ module is unitary if
\begin{equation*}
	X^*=-X,\quad\forall X\in\go.
\end{equation*}
Let us recall that
\begin{equation*}
	(iX)^*=(iX),\quad\forall X\in\go.
\end{equation*}
Now, we calculate $X^*$ for $X\in\g$. Relations \eqref{xyagp}--\eqref{xyagz}
show that
\begin{gather}
	H_\alpha^*=H_\alpha,\;H_\beta^*=H_\beta,\;H_\gamma^*=H_\gamma,\label{hst}\\
	X_\alpha^*=Y_\alpha,\;Y_\alpha^*=X_\alpha,\;
	X_\gamma^*=Y_\gamma,\;Y_\gamma^*=X_\gamma,\label{acst}\\
	X_\beta^*=-Y_\beta,\;Y_\beta^*=-X_\beta,\;
	X_{\alpha+\beta}^*=-Y_{\alpha+\beta},\;
	Y_{\alpha+\beta}^*=-X_{\alpha+\beta},\label{bst}\\
	X_{\beta+\gamma}^*=-Y_{\beta+\gamma},\;
	Y_{\beta+\gamma}^*=-X_{\beta+\gamma},
	X_{\alpha+\beta+\gamma}^*=-Y_{\alpha+\beta+\gamma},\;
	Y_{\alpha+\beta+\gamma}^*=-X_{\alpha+\beta+\gamma}\label{abcst}.
\end{gather}
Now, let us calculate $A_*^*$ and $B_*^*$ for our operators $A_*$ and $B_*$.
Firstly, let us analyze one interesting detail.

\begin{theorem}\label{thinpr}
	Let $v_{n,k,m}$ and $w_{p,r,q}$ be the highest weight vectors. If
	\begin{equation}
		(Y_\alpha^aY_\gamma^bv_{n,k,m},Y_\alpha^cY_\gamma^dw_{p,r,q})\neq0
	\end{equation}
	then $a=c$, $b=d$ and $(n,k,m)=(p,r,q)$
\end{theorem}

\begin{proof}
	Let us compare $n-2a$ and $p-2c$. Since
	\begin{gather*}
		(n-2a)(Y_\alpha^aY_\gamma^bv_{n,k,m},Y_\alpha^cY_\gamma^dw_{p,r,q})
		\stackrel{\eqref{sl2}}{=}
		(H_\alpha Y_\alpha^aY_\gamma^bv_{n,k,m},Y_\alpha^cY_\gamma^dw_{p,r,q})
		\stackrel{\eqref{hst}}{=}\\
		\stackrel{\eqref{hst}}{=}
		(Y_\alpha^aY_\gamma^bv_{n,k,m},H_\alpha Y_\alpha^cY_\gamma^dw_{p,r,q})
		\stackrel{\eqref{sl2}}{=}
		(p-2c)(Y_\alpha^aY_\gamma^bv_{n,k,m},Y_\alpha^cY_\gamma^dw_{p,r,q})
	\end{gather*}
	the condition produces $n-2a=p-2c$. If $a=c$ then $n=p$. If $a\neq c$ then
	we can assume that $a>c$. Since
	\begin{gather*}
		0\neq(Y_\alpha^{c+1}X_\alpha^{c+1}Y_\alpha^aY_\gamma^bv_{n,k,m},
		Y_\alpha^cY_\gamma^dw_{p,r,q})\stackrel{\eqref{acst}}{=}\\
		\stackrel{\eqref{acst}}{=}(X_\alpha^{c+1}Y_\alpha^aY_\gamma^bv_{n,k,m},
		X_\alpha^{c+1}Y_\alpha^cY_\gamma^dw_{p,r,q})=0
	\end{gather*}
	it follows $a=c$.
	Similarly, the action of $H_\gamma$ and our condition produce $m-2b=q-2d$,
	$b=d$ and $m=q$.
	Finally, the relation \eqref{hst} ($H_\beta^*=H_\beta$) produces
	\begin{equation*}
		(H_\beta Y_\alpha^aY_\gamma^bv_{n,k,m},Y_\alpha^cY_\gamma^dw_{p,r,q})=
		(Y_\alpha^aY_\gamma^bv_{n,k,m},H_\beta Y_\alpha^cY_\gamma^dw_{p,r,q}).
	\end{equation*}
	It follows that the action of $H_\beta$ produces the same value for both
	vectors. Since the action of $H_\alpha$ and $H_\gamma$ produce the
	same value, it follows $k=r$.
\end{proof}

\begin{remark}\label{innind}
	This theorem shows that the inner product of two vectors is 0 if vectors
	have different weight. It shows even more: the inner product is 0 if
	vectors belong to $K$-types of different highest weight vectors.
	It simplifies our construction of inner product since we have to consider
	the inner product on $K$-types of the same highest weight.
\end{remark}

\begin{proposition}\label{star}
	Adjoint operators are given by
	\begin{gather*}
		A_\beta^*=-B_\beta(H_\alpha+2)(H_\gamma+2),\;
		B_\beta^*=-A_\beta\frac1{(H_\alpha+1)(H_\gamma+1)},\\
		A_{\alpha+\beta}^*=-B_{\alpha+\beta}\frac{H_\gamma+2}{H_\alpha+1},\;
		B_{\alpha+\beta}^*=-A_{\alpha+\beta}\frac{H_\alpha+2}{H_\gamma+1},\\
		A_{\beta+\gamma}^*=-B_{\beta+\gamma}\frac{H_\alpha+2}{H_\gamma+1},\;
		B_{\beta+\gamma}^*=-A_{\beta+\gamma}\frac{H_\gamma+2}{H_\alpha+1},\\
		A_{\alpha+\beta+\gamma}^*=
			-B_{\alpha+\beta+\gamma}\frac1{(H_\alpha+1)(H_\gamma+1)},\;
		B_{\alpha+\beta+\gamma}^*=
			-A_{\alpha+\beta+\gamma}(H_\alpha+2)(H_\gamma+2).
	\end{gather*}
\end{proposition}

\begin{proof}
	Let us consider operators
	\begin{enumerate}
		\item $\ds C_{\alpha+\beta+\gamma}=X_{\alpha+\beta+\gamma}
			+X_{\beta+\gamma}X_\alpha\frac1{H_\alpha+2}\\
			\hspace*{18mm}-X_{\alpha+\beta}X_\gamma\frac1{H_\gamma+2}
			-X_\beta X_\alpha X_\gamma\frac1{(H_\alpha+2)(H_\gamma+2)}$
		\item $\ds C_{\alpha+\beta}=X_{\alpha+\beta}
			+Y_\gamma X_{\alpha+\beta+\gamma}\frac1{H_\gamma+1}
			+X_\beta X_\alpha\frac1{H_\alpha+2}$
		\item $\ds C_{\beta+\gamma}=X_{\beta+\gamma}
			-Y_\alpha X_{\alpha+\beta+\gamma}\frac1{H_\alpha+1}
			-X_\beta X_\gamma\frac1{H_\gamma+2}$
		\item $\ds C_{\beta}=X_{\beta}
			-Y_\alpha X_{\alpha+\beta}\frac1{H_\alpha+1}\\
			\hspace*{10mm}+Y_\gamma X_{\beta+\gamma}\frac1{H_\gamma+1}
			-Y_\alpha Y_\gamma X_{\alpha+\beta+\gamma}
			\frac1{(H_\alpha+1)(H_\gamma+1)}$
		\item $\ds D_{\alpha+\beta+\gamma}=Y_{\alpha+\beta+\gamma}
			+Y_\alpha Y_{\beta+\gamma}\frac1{H_\alpha+1}\\
			\hspace*{18mm}-Y_\gamma Y_{\alpha+\beta}\frac1{H_\gamma+1}
			-Y_\alpha Y_\gamma Y_\beta\frac1{(H_\alpha+1)(H_\gamma+1)}$
		\item $\ds D_{\alpha+\beta}=Y_{\alpha+\beta}
			+Y_\alpha Y_\beta\frac1{H_\alpha+1}
			+Y_{\alpha+\beta+\gamma}X_\gamma\frac1{H_\gamma+2}$
		\item $\ds D_{\beta+\gamma}=Y_{\beta+\gamma}
			-Y_\gamma Y_\beta\frac1{H_\gamma+1}
			-Y_{\alpha+\beta+\gamma}X_\alpha\frac1{H_\alpha+2}$
		\item $\ds D_\beta=Y_\beta
			-Y_{\alpha+\beta}X_\alpha\frac1{H_\alpha+2}\\
			\hspace*{10mm}+Y_{\beta+\gamma}X_\gamma\frac1{H_\gamma+2}
			-Y_{\alpha+\beta+\gamma}X_\alpha
			X_\gamma\frac1{(H_\alpha+2)(H_\gamma+2)}$
	\end{enumerate}
	We claim that
	\begin{equation*}
		C_{\alpha+\beta+\gamma}^*=-D_{\alpha+\beta+\gamma},
		C_{\alpha+\beta}^*=-D_{\alpha+\beta},
		C_{\beta+\gamma}^*=-D_{\beta+\gamma},
		C_\beta^*=-D_\beta.
	\end{equation*}
	One should notice that in each expression for $C_*$ and $D_*$, noncompact
	roots appear only once in each term. It produces a minus sign
	(by \eqref{bst}) and \eqref{abcst}).
	It remains to be clarified how $H_\alpha$ and $H_\gamma$ behave.
	Observe that the operators $H$ have been placed on the right-hand side of
	each expression for convenience. For example, it is also possible to write
	\begin{equation*}
		C_{\alpha+\beta}=X_{\alpha+\beta}
			+Y_\gamma X_{\alpha+\beta+\gamma}\frac1{H_\gamma+1}
			+\frac1{H_\alpha+1}X_\beta X_\alpha.
	\end{equation*}
	It remains to show that
	\begin{equation}\label{mult}
		X_\beta X_\alpha\frac1{H_\alpha+2}=\frac1{H_\alpha+1}X_\beta X_\alpha.
	\end{equation}
	It reduces to
	\begin{equation*}
		(H_\alpha+1)X_\beta X_\alpha=X_\beta X_\alpha(H_\alpha+2)
	\end{equation*}
	and it follows from the commutation relation
	\[ [H_\alpha,X_\beta X_\alpha]=X_\beta X_\alpha. \]
	Operators $C_*$ and $D_*$ act on the whole space $V$. Operators $A_*$ and
	$B_*$ act on the subspace of highest weights vectors of $K$-types only and
	the restriction of operators $C_*$ and $D_*$ to that space (and the multiplication by appropriate function in $H_\alpha$ and $H_\gamma$)
	produces operators $A_*$ and $B_*$.
\end{proof}

\begin{remark}
	Let us sketch another proof of that proposition.
	Let us choose one noncompact root, say $\alpha+\beta$.
	Let $v$ of the highest weight vector of the weight $\lambda$ and $w$ be
	the highest weight vector of weight $\mu$.
	Since $X_{\alpha+\beta}^*=-Y_{\alpha+\beta}$,
	\begin{equation}\label{xst}
		(X_{\alpha+\beta}v,w)=-(v,Y_{\alpha+\beta}w)
	\end{equation}
	By Proposition \ref{xyizab} and Theorem \ref{thinpr} (since $w$ is of
	the highest weight, the inner product of $w$ and vectors which are not of
	the highest weight is 0)
	\begin{equation*}
		(X_{\alpha+\beta}v,w)=
		\left(A_{\alpha+\beta}\frac1{H_\gamma+1}v,w\right).
	\end{equation*}
	Similarly,
	\begin{equation*}
		(v,Y_{\alpha+\beta}w)=\left(v,B_{\alpha+\beta}\frac1{H_\alpha+1}w\right)
	\end{equation*}
	Hence, the equation \eqref{xst} produces
	\begin{equation*}
		\left(A_{\alpha+\beta}\frac1{H_\gamma+1}v,w\right)=
		-\left(v,B_{\alpha+\beta}\frac1{H_\alpha+1}w\right)
	\end{equation*}
	and it follows
	\begin{equation}\label{den}
		\left(A_{\alpha+\beta}\frac1{H_\gamma+1}\right)^*=
		-B_{\alpha+\beta}\frac1{H_\alpha+1}
	\end{equation}
	or
	\begin{equation}\label{num}
		A_{\alpha+\beta}^*=-B_{\alpha+\beta}\frac{H_\gamma+2}{H_\alpha+1}.
	\end{equation}
	To see why $H_\gamma+2$ appears in the numerator, consider the following
	calculation: if the weight of $v$ is $(n,k,m)$, then the weight of
	$w$ is $(n+1,k+2,m-1)$ (if we want $(X_{\alpha+\beta}v,w)\neq0$). Then
	$H_\gamma+1$ in the denominator on the left hand side in \eqref{den}
	produces $m+1$ and $H_\gamma+2$ in the numerator on the right hand side in
	\eqref{num} produces also $m+1$. One should notice that in \eqref{mult},
	both side act on the same vector of the highest weight.
\end{remark}

\begin{proposition}
	The adjoint operators $Q^*$ and $R^*$ are given by
	\begin{eqnarray}\label{qrstar}
		Q^*=R\qquad\mbox{and}\qquad R^*=Q.
	\end{eqnarray}
\end{proposition}

\begin{proof}
	A straightforward calculation gives
	\begin{equation*}
		Q^*=(X_{\alpha+\beta}X_{\beta+\gamma}-X_\beta X_{\alpha+\beta+\gamma})^*
		=Y_{\alpha+\beta}Y_{\beta+\gamma}-Y_\beta Y_{\alpha+\beta+\gamma}=R.
	\end{equation*}
	We used that $X_\delta^*=-Y_\delta$, for $\delta\neq\alpha,\gamma$
	and that the operators appearing in products commute.
	The proof of other formula is analogous.
\end{proof}

\subsection{Coefficients in unitary $(\g,K)$ modules}

We have already defined certain coefficients by \eqref{coefdef} in the
Subsection \ref{coef}. From here on, these coefficients play a more significant
role. Again, let us assume that the highest weight vector $v$ has the weight
$(n,k,m)$ and there are no other highest weight vectors of that weight
(that weight occurs with multiplicity one). We already saw that
\begin{equation}\label{abdelta}
	A_\delta B_\delta v=\lambda(v,\delta) v.
\end{equation}
Let us put
\begin{equation}\label{badelta}
	B_\delta A_\delta v=\mu(v,\delta) v.
\end{equation}
Now, let us take one particular root, say
$\beta$ and calculate $A_\beta B_\beta v$. By Proposition \ref{star},
\begin{equation}\label{bstar}
	B_\beta^*=-A_\beta\frac1{(H_\alpha+1)(H_\gamma+1)}
\end{equation}
or
\begin{equation*}
	A_\beta=-B_\beta^*(H_\alpha+1)(H_\gamma+1)
\end{equation*}
Now, $A_\beta B_\beta v$ transforms to
\begin{equation*}
	A_\beta B_\beta v=-B_\beta^*(H_\alpha+1)(H_\gamma+1)B_\beta v=
	-B_\beta^*B_\beta(H_\alpha+2)(H_\gamma+2)v.
\end{equation*}
It follows from \eqref{abdelta}
\begin{gather*}
	\lambda(v,\beta)(v,v)=(\lambda(v,\beta)v,v)=(A_\beta B_\beta v,v)=\\
	=-(B_\beta^*B_\beta(H_\alpha+2)(H_\gamma+2)v,v)=
	-(n+2)(m+2)(B_\beta v,B_\beta v)
\end{gather*}
and
\begin{equation*}
	\lambda(v,\beta)=-\frac{(n+2)(m+2)(B_\beta v,B_\beta v)}{(v,v)}.
\end{equation*}
If follows that
\begin{equation}\label{bneg}
	\lambda(v,\beta)\in\R,\quad\lambda(v,\beta)\leq0.
\end{equation}
Let us rewrite the above expression in the form
\begin{equation}\label{bvbv}
	(B_\beta v,B_\beta v)=-\frac{\lambda(v,\beta)}{(n+2)(m+2)}(v,v).
\end{equation}
It shows that inner product $(B_\beta v,B_\beta v)$ is determined by values
$(v,v)$ and $\lambda(v,\beta)$.
Let us emphasize one obvious detail. The action of $B_\delta$ on \eqref{abdelta}
produces
\begin{equation*}
	B_\delta A_\delta(B_\delta v)=\lambda(v,\delta)(B_\delta v).
\end{equation*}
Hence, if we start with $B_\delta v\neq0$ and apply operators $A_\delta$ and
$B_\delta$, we obtain the same vector $B_\delta v$ multiplied by the same
coefficient $\lambda(v,\delta)$. Also, by \eqref{badelta}
\begin{equation*}
	B_\delta A_\delta(B_\delta v)=\lambda(v,\delta)(B_\delta v)
	=\mu(B_\delta v,\delta)(B_\delta v)
\end{equation*}
and hence ($B_\delta v\neq0$)
\begin{equation*}
	\lambda(v,\delta)=\mu(B_\delta v,\delta).
\end{equation*}
If $B_\delta v=0$, the norm identity gives $\lambda(v,\delta)=0$, but
$\mu(B_\delta v,\delta)$ is not defined and no coefficient may be propagated
through that zero transition. In what follows, coefficient propagation is used
only along nonzero transitions, which is sufficient for the generic case
considered here.

We proceed with the analogous calculation for
$B_{\alpha+\beta+\gamma}A_{\alpha+\beta+\gamma}$. Namely, the element
$A_{\alpha+\beta+\gamma}$ sends vectors of the highest weight $(n,k,m)$ to
vectors of the highest weight $(n+1,k+2,m+1)$. Elements $B_\beta$ sends vectors
of the highest weight $(n,k,m)$ to vectors of the highest weight
$(n+1,k-2,m+1)$. We will give a more detailed explanation soon.
By Proposition \ref{star},
\begin{equation}\label{astar}
	A_{\alpha+\beta+\gamma}^*=
	-B_{\alpha+\beta+\gamma}\frac1{(H_\alpha+1)(H_\gamma+1)}
\end{equation}
or
\begin{equation*}
	B_{\alpha+\beta+\gamma}=-A_{\alpha+\beta+\gamma}^*(H_\alpha+1)(H_\gamma+1)
\end{equation*}
Now, $B_{\alpha+\beta+\gamma}A_{\alpha+\beta+\gamma}v$ transforms to
\begin{equation*}
	B_{\alpha+\beta+\gamma}A_{\alpha+\beta+\gamma}v=
	-A_{\alpha+\beta+\gamma}^*(H_\alpha+1)(H_\gamma+1)
		A_{\alpha+\beta+\gamma}v=
	-A_{\alpha+\beta+\gamma}^*A_{\alpha+\beta+\gamma}(H_\alpha+2)(H_\gamma+2)v.
\end{equation*}
It follows from \eqref{abdelta}
\begin{gather*}
	\mu(v,\alpha+\beta+\gamma)(v,v)=(\mu(v,\alpha+\beta+\gamma)v,v)=
	(B_{\alpha+\beta+\gamma}A_{\alpha+\beta+\gamma}v,v)=\\
	=-(A_{\alpha+\beta+\gamma}^*A_{\alpha+\beta+\gamma}(H_\alpha+2)(H_\gamma+2)
		v,v)=
	-(n+2)(m+2)(A_{\alpha+\beta+\gamma}v,A_{\alpha+\beta+\gamma}v)
\end{gather*}
and
\begin{equation*}
	\mu(v,\alpha+\beta+\gamma)=
	-\frac{(n+2)(m+2)(A_{\alpha+\beta+\gamma} v,A_{\alpha+\beta+\gamma}v)}
	{(v,v)}.
\end{equation*}
If follows that
\begin{equation}\label{abcneg}
	\mu(v,\alpha+\beta+\gamma)\in\R,\quad\mu(v,\alpha+\beta+\gamma)\leq0.
\end{equation}
Again, we can write
\begin{equation}\label{abcvabcv}
	(A_{\alpha+\beta+\gamma} v,A_{\alpha+\beta+\gamma}v)
	=-\frac{\mu(v,\alpha+\beta+\gamma)}{(n+2)(m+2)}(v,v).
\end{equation}
It shows that the inner product
$(A_{\alpha+\beta+\gamma} v,A_{\alpha+\beta+\gamma}v)$ is determined by
values $\mu(v,\alpha+\beta+\gamma)$ and $(v,v)$.

\begin{remark}
	Let $S$ be the set of $K$ types. $K$ types are denoted by our triples
	$(n,k,m)$ where $n,m\geq0$ and satisfy certain parity conditions.
	Then $S$ is a subset of 2 octants. In that set $S$ there is a subset
	$T$ for which the sum $n+m$ is minimal.
	In the conditional discussion we assume that every $K$-type in $T$ has
	multiplicity one. This assumption reflects the expected multiplicity-one
	behavior on the boundary of the $K$-support and is consistent with the
	standard multiplicity formulas for irreducible admissible representations.
	Our purpose here is to analyze the resulting boundary coefficient system; in
	particular, all coefficients considered on $T$ are genuine scalars rather
	than endomorphisms of higher-dimensional multiplicity spaces.
	Operators $A_{\alpha+\beta}$, $B_{\alpha+\beta}$,
	$A_{\beta+\gamma}$ and $B_{\beta+\gamma}$ leave the vectors from 
	$T$ inside $T$. The action of operators $A_\beta$ and
	$B_{\alpha+\beta+\gamma}$ decrease the sum $n+m$. Since, the sum
	$n+m$ is minimal for vectors in $T$, the action of these two operators is 0.
	Hence, only operators $B_\beta$ and $A_{\alpha+\beta+\gamma}$ act on the
	subspace $T$ and produce a nonzero vector outside of $T$. This is why the
	analysis of these two operators was carried out above. We said that we
	can put $(v,v)=1$ for some chosen vectors in $T$. Note
	that in formulas \eqref{bvbv} and \eqref{abcvabcv} the same expression
	$(n+2)(m+2)$ appears in the denominator. Hence, we can construct the inner
	product without that expression. It looks as an artificial construction,
	but we could use operators $C_*$ and $D_*$ instead of $A_*$ and $B_*$
	(actually, it is more natural and more similar to our operators $X_*$
	and $Y_*$, but it contains the denominator). According to
	Remark \ref{innind}, the inner product is well defined. Also, it is
	clear that the key relation is
	\begin{equation}\label{dneg}
		\lambda(v,\delta),\mu(v,\delta)\in\R,\quad
		\lambda(v,\delta),\mu(v,\delta)\leq0.
	\end{equation}
	for noncompact roots $\delta$. We have proved it for
	$\lambda(v,\beta)$ and $\mu(v,\alpha+\beta+\gamma)$.
	The proof for all other coefficients is the same.
	These observations give necessary local norm relations; they do not by
	themselves prove path independence of the coefficient propagation or provide
	an independent construction of a global invariant inner product. However,
	for the unitary representations under consideration the existence of such a
	global invariant inner product is already guaranteed by the known
	classification. Thus the issue here is not the existence of the invariant
	form, but rather to describe it explicitly in terms of the boundary
	coefficients and the operators introduced above.
\end{remark}

This Remark demonstrates our approach. We will start with the subset $T$
(which is equal to one point for holomorphic series) and from that set we
reconstruct the inner product on the whole $(\g,K)$ module $V$. On the subset
$T$ multiplicities of $K$ types are 1 and it is easy to construct the inner
product according to Theorem \ref{thinpr}. After that we construct the
inner product on the neighborhood of $T$. The key observation are relations
\eqref{bneg}, \eqref{abcneg} and appropriate expressions for remaining two
noncompact roots: $\alpha+\beta$ and $\beta+\gamma$. The construction becomes
more subtle when the multiplicities exceed 1. Also, the definition of $T$ is
not very precise.
Nevertheless, a specific computation we carried out is very promising.
Hence, we have

\begin{conjecture}
	Unitary (\g,$K$) modules are determined by the set $T$ and certain
	additional parameters.
\end{conjecture}

\begin{remark}
	The statement of the Conjecture is not meant as a precise parametrization.
	The unitary dual of $SU(2,2)$ is already known to be parametrized. Our aim
	is to describe such representations from a different point of view. In the
	remainder of the paper we develop relations among the coefficients
	associated with the $K$-types in $T$ and give examples illustrating how
	these relations encode the corresponding $K$-type structure.
\end{remark}

\subsection{Calculation of coefficients}

\begin{lemma}
	The following formulas hold
	\begin{equation}\label{commfun}
		f(H_\alpha)A_{\alpha+\beta+\gamma}=A_{\alpha+\beta+\gamma}f(H_\alpha+1)
	\end{equation}
	for any function $f(H_\alpha)$. The similar formulas hold for all other
	roots and for $H_\gamma$.
\end{lemma}

\begin{proof}
	Let $v$ be a highest weight vector of weight $(n,k,m)$. Then
	$A_{\alpha+\beta+\gamma}v$ has the weight $(n+1,k+2,m+1)$ and
	$f(H_\alpha)$ acts as a multiplication by $f(n+1)$. Hence, we multiply by 
	$f(n+1)$; consequently, $f(H_\alpha+1)$ appears on the right-hand side.
\end{proof}

\begin{example}
	For $f(H_\alpha)=\frac1{H_\alpha+1}$, formula \eqref{commfun} transforms to
	\begin{equation*}
		\frac1{H_\alpha+1}A_{\alpha+\beta+\gamma}=
		A_{\alpha+\beta+\gamma}\frac1{H_\alpha+2}.
	\end{equation*}
	It looks unnatural. However both sides of this formula acting on
	the highest weight vector of weight $(n,k,m)$ produce
	$\frac1{n+2}A_{\alpha+\beta+\gamma}v$ since the weight of
	$A_{\alpha+\beta+\gamma}v$ is $(n+1,k+2,m+1)$.
\end{example}

\begin{proposition}\label{system}
	Let $v$ be a highest weight vector of weight $(n,k,m)$.
	Then the following system of 4 equations holds:
	\begin{flushleft}
		$\ds \frac1{(n+2)(m+2)}A_\beta B_\beta
			-\frac1{(n+1)(m+1)}B_\beta A_\beta
			-\frac1{(n+1)(n+2)(m+1)}B_{\alpha+\beta}A_{\alpha+\beta}-$
	\end{flushleft}
	\begin{flushright}
		$\ds -\frac1{(n+1)(m+1)(m+2)}B_{\beta+\gamma}A_{\beta+\gamma}
			-\frac1{(n+1)(n+2)(m+1)(m+2)}
			B_{\alpha+\beta+\gamma}A_{\alpha+\beta+\gamma}=\frac12(k-n-m)$
	\end{flushright}
	\begin{flushleft}
		$\ds \frac1{(n+1)(m+2)}A_{\alpha+\beta}B_{\alpha+\beta}
			-\frac1{(n+2)(m+1)}B_{\alpha+\beta}A_{\alpha+\beta}+$
	\end{flushleft}
	\begin{flushright}
		$\ds +\frac1{(n+1)(n+2)(m+2)}A_\beta B_\beta-\frac1{(n+2)(m+1)(m+2)}
			B_{\alpha+\beta+\gamma}A_{\alpha+\beta+\gamma}=\frac12(k+n-m)$
	\end{flushright}
	\begin{flushleft}
		$\ds \frac1{(n+2)(m+1)}A_{\beta+\gamma}B_{\beta+\gamma}
			-\frac1{(n+1)(m+2)}B_{\beta+\gamma}A_{\beta+\gamma}+$
	\end{flushleft}
	\begin{flushright}
		$\ds +\frac1{(n+2)(m+1)(m+2)}A_\beta B_\beta-\frac1{(n+1)(n+2)(m+2)}
			B_{\alpha+\beta+\gamma}A_{\alpha+\beta+\gamma}=\frac12(k-n+m)$
	\end{flushright}
	\begin{flushleft}
		$\ds \frac1{(n+1)(m+1)}A_{\alpha+\beta+\gamma}B_{\alpha+\beta+\gamma}
			-\frac1{(n+2)(m+2)}B_{\alpha+\beta+\gamma}A_{\alpha+\beta+\gamma}
			+\frac1{(n+1)(m+1)(m+2)}A_{\alpha+\beta}B_{\alpha+\beta}+$
	\end{flushleft}
	\begin{flushright}
		$\ds +\frac1{(n+1)(n+2)(m+1)}A_{\beta+\gamma}B_{\beta+\gamma}
			+\frac1{(n+1)(n+2)(m+1)(m+2)}A_\beta B_\beta=\frac12(k+n+m)$
	\end{flushright}
\end{proposition}

\begin{remark}
	We wrote $n$ instead of $H_\alpha$ to emphasize that these four relations
	form a system of four equations in eight unknowns. If the multiplicity of
	the weight  $(n,k,m)$ is equal to one and some of the unknowns are known,
	the system can be solved.
	For simplicity, we wrote $\frac12(k-n-m)$ instead of $\frac12(k-n-m)I$, although the latter is slightly more precise since $I$ denotes the identity operator.
	One detail should be noticed. If the multiplicity is bigger than 1, the
	expressions $A_*B_*v$ and $B_*A_*v$ do not have to be multiples of $v$.
	However, the left-hand side of each equation is a multiple of $v$ according
	to the right-hand side. It means that all "additional pieces" cancel out
	on the left-hand side of each equation.
\end{remark}

\begin{proof}
	We will prove only the first equation, since the proofs of the remaining
	three equations are analogous.
	We start with the left-hand side of the equation, then express operators
	$A_*$ and $B_*$ in terms of $X_*$ and $Y_*$ and write all terms in the
	form $YXH$. In the end, most of the terms will cancel out.
	\begin{flushleft}
		$\ds \frac1{(n+2)(m+2)}A_\beta B_\beta
		-\frac1{(n+1)(m+1)}B_\beta A_\beta
		-\frac1{(n+1)(n+2)(m+1)}B_{\alpha+\beta}A_{\alpha+\beta}$
	\end{flushleft}
	\begin{flushright}
		$\ds -\frac1{(n+1)(m+1)(m+2)}B_{\beta+\gamma}A_{\beta+\gamma}
		-\frac1{(n+1)(n+2)(m+1)(m+2)}
		B_{\alpha+\beta+\gamma}A_{\alpha+\beta+\gamma}=$
	\end{flushright}
	\begin{flushleft}
		$\ds =\left(X_\beta-Y_\alpha X_{\alpha+\beta}\frac1{n+1}+
		Y_\gamma X_{\beta+\gamma}\frac1{m+1}-Y_\alpha Y_\gamma
		X_{\alpha+\beta+\gamma}\frac1{(n+1)(m+1)}\right)Y_\beta-$
	\end{flushleft}
	\begin{flushright}
		$\ds -Y_\beta\left(X_\beta-Y_\alpha X_{\alpha+\beta}\frac1{n+1}+
		Y_\gamma X_{\beta+\gamma}\frac1{m+1}-Y_\alpha Y_\gamma
		X_{\alpha+\beta+\gamma}\frac1{(n+1)(m+1)}\right)-$
	\end{flushright}
	\begin{flushright}
		$\ds -\left(Y_{\alpha+\beta}+Y_\alpha Y_\beta\frac1{n+1}\right)
		\left(X_{\alpha+\beta}+Y_\gamma X_{\alpha+\beta+\gamma}
		\frac1{m+1}\right)\frac1{n+1}-$
	\end{flushright}
	\begin{flushright}
		$\ds -\left(Y_{\beta+\gamma}-Y_\gamma Y_\beta\frac1{m+1}\right)
		\left(X_{\beta+\gamma}-Y_\alpha X_{\alpha+\beta+\gamma}
		\frac1{n+1}\right)\frac1{m+1}-$
	\end{flushright}
	\begin{flushright}
		$\ds -\left(Y_{\alpha+\beta+\gamma}+
		Y_\alpha Y_{\beta+\gamma}\frac1{n+1}-Y_\gamma Y_{\alpha+\beta}
		\frac1{m+1}-Y_\alpha Y_\gamma Y_\beta\frac1{(n+1)(m+1)}\right)
		X_{\alpha+\beta+\gamma}\frac1{(n+1)(m+1)}=$
	\end{flushright}
	\begin{flushleft}
		$\ds =X_\beta Y_\beta-Y_\alpha Y_\beta X_{\alpha+\beta}\frac1{n+2}+
		Y_\gamma Y_\beta X_{\beta+\gamma}\frac1{m+2}-Y_\alpha Y_\gamma Y_\beta
		X_{\alpha+\beta+\gamma}\frac1{(n+2)(m+2)}-$
	\end{flushleft}
	\begin{flushright}
		$\ds -Y_\beta X_\beta+Y_\alpha Y_\beta X_{\alpha+\beta}\frac1{n+1}
		+Y_{\alpha+\beta}X_{\alpha+\beta}\frac1{n+1}-
		Y_\gamma Y_\beta X_{\beta+\gamma}\frac1{m+1}
		+Y_{\beta+\gamma}X_{\beta+\gamma}\frac1{m+1}+$
	\end{flushright}
	\begin{flushright}
		$\ds +\left(Y_\alpha Y_\gamma Y_\beta-Y_\alpha Y_{\beta+\gamma}
		+Y_\gamma Y_{\alpha+\beta}-Y_{\alpha+\beta+\gamma}\right)
		X_{\alpha+\beta+\gamma}\frac1{(n+1)(m+1)}-$
	\end{flushright}
	\begin{flushright}
		$\ds -Y_{\alpha+\beta}X_{\alpha+\beta}\frac1{n+1}
		-Y_\alpha Y_\beta X_{\alpha+\beta}\frac1{(n+1)(n+2)}-$
	\end{flushright}
	\begin{flushright}
		$\ds-Y_\gamma Y_{\alpha+\beta} X_{\alpha+\beta+\gamma}\frac1{(n+1)(m+1)}
		+Y_{\alpha+\beta+\gamma}X_{\alpha+\beta+\gamma}\frac1{(n+1)(m+1)}-$
	\end{flushright}
	\begin{flushright}
		$\ds -Y_\alpha Y_\gamma Y_\beta X_{\alpha+\beta+\gamma}
		\frac1{(n+1)(n+2)(m+1)}
		+Y_\alpha Y_{\beta+\gamma}X_{\alpha+\beta+\gamma}
		\frac1{(n+1)(n+2)(m+1)}-$
	\end{flushright}
	\begin{flushright}
		$\ds -Y_{\beta+\gamma}X_{\beta+\gamma}\frac1{m+1}
		+Y_\gamma Y_\beta X_{\beta+\gamma}\frac1{(m+1)(m+2)}+$
	\end{flushright}
	\begin{flushright}
		$\ds+Y_\alpha Y_{\beta+\gamma} X_{\alpha+\beta+\gamma}\frac1{(n+1)(m+1)}
		+Y_{\alpha+\beta+\gamma}X_{\alpha+\beta+\gamma}\frac1{(n+1)(m+1)}-$
	\end{flushright}
	\begin{flushright}
		$\ds -Y_\alpha Y_\gamma Y_\beta X_{\alpha+\beta+\gamma}
		\frac1{(n+1)(m+1)(m+2)}
		-Y_\gamma Y_{\alpha+\beta}X_{\alpha+\beta+\gamma}
		\frac1{(n+1)(m+1)(m+2)}-$
	\end{flushright}
	\begin{flushright}
		$\ds -Y_{\alpha+\beta+\gamma}X_{\alpha+\beta+\gamma}\frac1{(n+1)(m+1)}
		-Y_\alpha Y_{\beta+\gamma}X_{\alpha+\beta+\gamma}
		\frac1{(n+1)(n+2)(m+1)}+$
	\end{flushright}
	\begin{flushright}
		$\ds Y_\gamma Y_{\alpha+\beta}
		X_{\alpha+\beta+\gamma}\frac1{(n+1)(m+1)(m+2)}
		+Y_\alpha Y_\gamma Y_\beta
		X_{\alpha+\beta+\gamma}\frac1{(n+1)(n+2)(m+1)(m+2)}=$
	\end{flushright}
	Now, let us collect like terms.
	The term $Y_{\alpha+\beta+\gamma}X_{\alpha+\beta+\gamma}\frac1{(n+1)(m+1)}$
	appears four times, twice with a plus sign, twice with a minus sign
	and therefore cancels out.
	The term $Y_\alpha Y_\beta X_{\alpha+\beta}$ appears three times, with
	coefficients $-\frac1{n+2}$, $\frac1{n+1}$ and $-\frac1{(n+1)(n+2)}$
	and therefore cancels out.
	The term $Y_\gamma Y_\beta X_{\beta+\gamma}$ appears three times, with
	coefficients $\frac1{m+2}$, $-\frac1{m+1}$ and $\frac1{(m+1)(m+2)}$
	and therefore cancels out.
	The term $Y_{\alpha+\beta}X_{\alpha+\beta}\frac1{n+1}$
	appears twice, once with a plus sign, once with a minus sign
	and therefore cancels out.
	The term $Y_{\beta+\gamma}X_{\beta+\gamma}\frac1{m+1}$
	appears twice, once with a plus sign, once with a minus sign
	and therefore cancels out.
	The term $Y_\alpha Y_{\beta+\gamma}X_{\alpha\beta+\gamma}\frac1{(n+1)(m+1)}$
	appears twice, once with a minus sign, once with a plus sign
	and therefore cancels out.
	The term $Y_\gamma Y_{\alpha+\beta}X_{\alpha\beta+\gamma}\frac1{(n+1)(m+1)}$
	appears twice, once with a minus sign, once with a plus sign
	and therefore cancels out.
	The term $Y_\alpha Y_{\beta+\gamma}X_{\alpha\beta+\gamma}
	\frac1{(n+1)(n+2)(m+1)}$ appears twice, once with a plus sign, once with
	a minus sign and therefore cancels out.
	The term $Y_\gamma Y_{\alpha+\beta}X_{\alpha\beta+\gamma}
	\frac1{(n+1)(m+1)(m+2)}$ appears twice, once with a minus sign, once with
	a plus sign and therefore cancels out.
	The term $Y_\alpha Y_\gamma Y_\beta X_{\alpha\beta+\gamma}$ appears with
	coefficient $-\frac1{(n+2)(m+2)}$, $\frac1{(n+1)(m+1)}$,
	$-\frac1{(n+1)(n+2)(m+1)}$, $-\frac1{(n+1)(m+1)(m+2)}$ and
	$\frac1{(n+1)(n+2)(m+1)(m+2)}$ and their sum is equal to 0. It remains
	\begin{equation*}
		=X_\beta Y_\beta-Y_\beta X_\beta=H_\beta=\frac12(k-n-m)
	\end{equation*}
	and it proves the first formula.
\end{proof}

\begin{remark}\label{4to2}
	If $n=0$ then $B_{\beta+\gamma}A_{\beta+\gamma}=0$,
	$A_{\alpha+\beta}B_{\alpha+\beta}=0$,
	$A_{\alpha+\beta+\gamma}B_{\alpha+\beta+\gamma}=0$ and $B_\beta A_\beta=0$.
	It looks like that the situation is very determined. We have a system of
	four equations in four unknowns. However, our system reduces to
	\begin{flushleft}
		$\ds \frac1{2(m+2)}A_\beta B_\beta
		-\frac1{2(m+1)}B_{\alpha+\beta}A_{\alpha+\beta}-
		\frac1{2(m+1)(m+2)}
		B_{\alpha+\beta+\gamma}A_{\alpha+\beta+\gamma}=\frac12(k-m)$
	\end{flushleft}
	\begin{flushleft}
		$\ds -\frac1{2(m+1)}B_{\alpha+\beta}A_{\alpha+\beta}
		+\frac1{2(m+2)}A_\beta B_\beta-\frac1{2(m+1)(m+2)}
		B_{\alpha+\beta+\gamma}A_{\alpha+\beta+\gamma}=\frac12(k-m)$
	\end{flushleft}
	\begin{flushleft}
		$\ds \frac1{2(m+1)}A_{\beta+\gamma}B_{\beta+\gamma}
		+\frac1{2(m+1)(m+2)}A_\beta B_\beta-\frac1{2(m+2)}
		B_{\alpha+\beta+\gamma}A_{\alpha+\beta+\gamma}=\frac12(k+m)$
	\end{flushleft}
	\begin{flushleft}
		$\ds -\frac1{2(m+2)}B_{\alpha+\beta+\gamma}A_{\alpha+\beta+\gamma}
		+\frac1{2(m+1)}A_{\beta+\gamma}B_{\beta+\gamma}
			+\frac1{2(m+1)(m+2)}A_\beta B_\beta=\frac12(k+m)$.
	\end{flushleft}
	We see that the first two equations are equal and the last two equations
	are also equal. Hence, we have a system of two equations in four
	unknowns,
	\begin{equation*}
		\frac1{m+2}A_\beta B_\beta
		-\frac1{(m+1)(m+2)}B_{\alpha+\beta+\gamma}A_{\alpha+\beta+\gamma}
		-\frac1{m+1}B_{\alpha+\beta}A_{\alpha+\beta}=k-m
	\end{equation*}
	\begin{equation*}
		\frac1{(m+1)(m+2)}A_\beta B_\beta
		-\frac1{m+2}B_{\alpha+\beta+\gamma}A_{\alpha+\beta+\gamma}
		+\frac1{m+1}A_{\beta+\gamma}B_{\beta+\gamma}=k+m
	\end{equation*}
	Similarly if $m=0$ then $A_{\beta+\gamma}B_{\beta+\gamma}=0$,
	$B_{\alpha+\beta}A_{\alpha+\beta}=0$,
	$A_{\alpha+\beta+\gamma}B_{\alpha+\beta+\gamma}=0$ and $B_\beta A_\beta=0$.
	Our system reduces to
	\begin{flushleft}
		$\ds \frac1{(n+2)2}A_\beta B_\beta
			-\frac1{(n+1)2}B_{\beta+\gamma}A_{\beta+\gamma}-\frac1{(n+1)(n+2)2}
		B_{\alpha+\beta+\gamma}A_{\alpha+\beta+\gamma}=\frac12(k-n)$
	\end{flushleft}
	\begin{flushleft}
		$\ds \frac1{(n+1)2}A_{\alpha+\beta}B_{\alpha+\beta}
			+\frac1{(n+1)(n+2)2}A_\beta B_\beta-\frac1{(n+2)2}
		B_{\alpha+\beta+\gamma}A_{\alpha+\beta+\gamma}=\frac12(k+n)$
	\end{flushleft}
	\begin{flushleft}
		$\ds -\frac1{(n+1)2}B_{\beta+\gamma}A_{\beta+\gamma}
			+\frac1{(n+2)2}A_\beta B_\beta-\frac1{(n+1)(n+2)2}
			B_{\alpha+\beta+\gamma}A_{\alpha+\beta+\gamma}=\frac12(k-n)$
	\end{flushleft}
	\begin{flushleft}
		$\ds -\frac1{(n+2)2}B_{\alpha+\beta+\gamma}A_{\alpha+\beta+\gamma}
			+\frac1{(n+1)2}A_{\alpha+\beta}B_{\alpha+\beta}
			+\frac1{(n+1)(n+2)2}A_\beta B_\beta=\frac12(k+n)$
	\end{flushleft}
	Again, the first and the third equations are the same and the second and
	the fourth equations are the same. Hence, we have a system of two
	equations in four unknowns,
	\begin{equation*}
		\frac1{n+2}A_\beta B_\beta
		-\frac1{(n+1)(n+2)}B_{\alpha+\beta+\gamma}A_{\alpha+\beta+\gamma}
		-\frac1{n+1}B_{\beta+\gamma}A_{\beta+\gamma}=k-n
	\end{equation*}
	\begin{equation*}
		\frac1{(n+1)(n+2)}A_\beta B_\beta
		-\frac1{n+2}B_{\alpha+\beta+\gamma}A_{\alpha+\beta+\gamma}
		+\frac1{n+1}A_{\alpha+\beta}B_{\alpha+\beta}=k+n
	\end{equation*}
	Finally, if $n=m=0$ then all terms except $A_\beta B_\beta$ and
	$B_{\alpha+\beta+\gamma}A_{\alpha+\beta+\gamma}$ vanish and our system
	reduces to
	\begin{equation}\label{nmnula}
		A_\beta B_\beta-B_{\alpha+\beta+\gamma}A_{\alpha+\beta+\gamma}=2k
	\end{equation}
\end{remark}

\begin{proposition}
	Let $v$ be a highest weight vector of weight $(n,k,m)$. Then
	\begin{equation}\label{bad}
		B_{\alpha+\beta+\gamma}A_{\alpha+\beta+\gamma}Rv=
		RB_{\alpha+\beta+\gamma}A_{\alpha+\beta+\gamma}v
		+R(H_\alpha+H_\gamma+2)(H_\alpha+H_\beta+H_\gamma+1)v
	\end{equation}
	for $v$ satisfying $B_{\alpha+\beta+\gamma}v=0$
	and
	\begin{equation}\label{abu}
		A_\beta B_\beta Qv=QA_\beta B_\beta v+Q(H_\alpha+H_\gamma+2)(1-H_\beta)v
	\end{equation}
	for $v$ satisfying $A_\beta v=0$.
\end{proposition}

\begin{remark}\label{othcom}
Similar formulas can be derived for $A_\beta B_\beta Rv$ and
	$B_{\alpha+\beta+\gamma}A_{\alpha+\beta+\gamma}Qv$, but they will not be
	needed.
\end{remark}

\begin{proof}
	We start with
	\begin{equation*}
		[B_{\alpha+\beta+\gamma}A_{\alpha+\beta+\gamma},R]v=
		[B_{\alpha+\beta+\gamma},R]A_{\alpha+\beta+\gamma}v
		+B_{\alpha+\beta+\gamma}[A_{\alpha+\beta+\gamma},R]v.
	\end{equation*}
	Since $[B_{\alpha+\beta+\gamma},R]=0$ (by Proposition \ref{commabqr}) and
	$[A_{\alpha+\beta+\gamma},R]=-B_\beta(H_\alpha+H_\beta+H_\gamma+1)$ (again,
	by Proposition \ref{commabqr}),
	\begin{equation*}
		[B_{\alpha+\beta+\gamma}A_{\alpha+\beta+\gamma},R]v=
		-B_{\alpha+\beta+\gamma}B_\beta(H_\alpha+H_\beta+H_\gamma+1)v.
	\end{equation*}
	By Proposition \ref{commab}
	(and our assumption $B_{\alpha+\beta+\gamma}v=0$),
	\begin{equation*}
		R(H_\alpha+H_\gamma+2)=[B_\beta,B_{\alpha+\beta+\gamma}]=
		B_\beta B_{\alpha+\beta+\gamma}v-B_{\alpha+\beta+\gamma}B_\beta v=
		-B_{\alpha+\beta+\gamma}B_\beta v
	\end{equation*}
	and it produces
	\begin{equation*}
		[B_{\alpha+\beta+\gamma}A_{\alpha+\beta+\gamma},R]v=
		R(H_\alpha+H_\gamma+2)(H_\alpha+H_\beta+H_\gamma+1)v.
	\end{equation*}
	Now, \eqref{bad} follows. For the second formula we start with
	\begin{equation*}
		[A_\beta B_\beta,Q]v=[A_\beta,Q]B_\beta v+A_\beta[B_\beta,Q]v.
	\end{equation*}
	By Proposition \ref{commabqr} $[A_\beta,Q]=0$ and
	$[B_\beta,Q]=A_{\alpha+\beta+\gamma}(H_\beta-1)$. Hence
	\begin{equation*}
		[A_\beta B_\beta,Q]v=A_\beta A_{\alpha+\beta+\gamma}(H_\beta-1).
	\end{equation*}
	By Proposition \ref{commab} and our assumption $A_\beta v=0$,
	\begin{equation*}
		A_\beta A_{\alpha+\beta+\gamma}(H_\beta-1)v=
		(A_\beta A_{\alpha+\beta+\gamma}-A_{\alpha+\beta+\gamma}A_\beta)
		(H_\beta-1)v=-Q(H_\alpha+H_\gamma+2)(H_\beta-1)v
	\end{equation*}
	Hence,
	\begin{equation*}
		[A_\beta B_\beta,Q]v=Q(H_\alpha+H_\gamma+2)(1-H_\beta)v
	\end{equation*}
	and \eqref{abu} follows.
\end{proof}

\begin{proposition}\label{abbcmnul}
	Let $v$ be a highest weight vector of weight $(n,k,0)$. Then
	\begin{equation*}
		B_{\beta+\gamma}A_{\beta+\gamma}Rv=
		RB_{\beta+\gamma}A_{\beta+\gamma}v+RH_\alpha H_\beta v
	\end{equation*}
	and
	\begin{equation*}
		A_{\alpha+\beta}B_{\alpha+\beta}Qv=
		QA_{\alpha+\beta}B_{\alpha+\beta}v-QH_\alpha(H_\alpha+H_\beta) v
	\end{equation*}
	Similar relations can be written for $B_{\beta+\gamma}A_{\beta+\gamma}Qv$
	and $A_{\alpha+\beta}B_{\alpha+\beta}Rv$.
\end{proposition}

\begin{proof}
	It suffices to prove the first relation, as the second follows by a similar
	argument. By Propositions \ref{commabqr} and \ref{commab}
	(and since $B_{\beta+\gamma}v=0$)
	\begin{equation*}
		[B_{\beta+\gamma}A_{\beta+\gamma},R]v=
		B_{\beta+\gamma}[A_{\beta+\gamma},R]v=
		B_{\beta+\gamma}B_{\alpha+\beta}H_\beta v=
		[B_{\beta+\gamma},B_{\alpha+\beta}]H_\beta v=
		RH_\alpha H_\beta v.
	\end{equation*}
	Now, the first relation follows.
\end{proof}

We omit the proof of the following proposition, since it is analogous to the
proof of the previous proposition.

\begin{proposition}\label{abbcnnul}
	Let $v$ be a highest weight vector of weight $(0,k,m)$. Then
	\begin{equation*}
		B_{\alpha+\beta}A_{\alpha+\beta}Rv=
		RB_{\alpha+\beta}A_{\alpha+\beta}v+RH_\gamma H_\beta v
	\end{equation*}
	and
	\begin{equation*}
		A_{\beta+\gamma}B_{\beta+\gamma}Qv=
		QA_{\beta+\gamma}B_{\beta+\gamma}v-QH_\gamma(H_\beta+H_\gamma) v
	\end{equation*}
	Similar relations can be written for $B_{\alpha+\beta}A_{\alpha+\beta}Qv$
	and $A_{\beta+\gamma}B_{\beta+\gamma}Rv$.
\end{proposition}

Now, let us investigate our coefficients $A_{\alpha+\beta}B_{\alpha+\beta}$,
$B_{\alpha+\beta}A_{\alpha+\beta}$, $A_{\beta+\gamma}B_{\beta+\gamma}$ and
$B_{\beta+\gamma}A_{\beta+\gamma}$ in another special situation.
Let us consider a $(\g,K)$ module whose highest weights $(n,k,m)$ all satisfy
\begin{equation}\label{bign}
	n+m\geq N
\end{equation}
for some $N\geq2$. We know that, in that case, the multiplicity of
every highest weight satisfying $n+m=N$ is equal to 1. Hence, our coefficients
are well defined.

Propositions \ref{abbcmnul} and \ref{abbcnnul} tells us that one coefficient,
say, $A_{\alpha+\beta}B_{\alpha+\beta}$, for $n=N$ and $m=0$ determines all
other coefficients satisfying $n=M$ and $m=0$. Now, let us go inside of the
region satisfying $n+m=N$.

We begin with a technical lemma. This lemma also provides a convenient choice of
vectors.

\begin{lemma}
	Let $u$ be a highest weight vector of weight $(N,k,0)$. Then
	\begin{equation}\label{npmbasis}
		A_{\alpha+\beta}A_{\beta+\gamma}^ju=(N-j+1)jA_{\beta+\gamma}^{j-1}Qu.
	\end{equation}
	for $j\in\{0,\ldots,N\}$.
\end{lemma}

\begin{remark}
	Hence, our choice of vectors have the form $A_{\beta+\gamma}^jQ^lu$ for
	$j\in\{0,\ldots,N\}$ and $l\in\{0,1\}$. Since $A_{\beta+\gamma}$ and $Q$
	commute, this choice is very nice. In this Lemma, the action of the
	operator $A_{\alpha+\beta}$ is computed.
\end{remark}

\begin{proof}
	The proof is by induction on $j$. For $j=0$, the statement becomes
	\begin{equation*}
		A_{\alpha+\beta}u=0.
	\end{equation*}
	This is clear since the weight $u$ is $(N,k,0)$.
	Assume that the statement holds for $j$, and let us prove it for $j+1$.
	\begin{equation*}
		A_{\alpha+\beta}A_{\beta+\gamma}^{j+1}u=
		(A_{\alpha+\beta}A_{\beta+\gamma}-A_{\beta+\gamma}A_{\alpha+\beta}
		+A_{\beta+\gamma}A_{\alpha+\beta})A_{\beta+\gamma}^ju=
	\end{equation*}
	By Proposition \ref{commab},
	$[A_{\alpha+\beta},A_{\beta+\gamma}]=Q(H_\alpha-H_\gamma)=(N-j-j)Q=(N-2j)Q$
	and for the third term, we apply the induction hypothesis. Hence
	\begin{equation*}
		((N-2j)+(N-j+1)j)A_{\beta+\gamma}^jQu=
		(N(j+1)-2j-j^2+j)A_{\beta+\gamma}^jQu=
		(N-j)(j+1)A_{\beta+\gamma}^jQu
	\end{equation*}
	and this completes the induction step.
\end{proof}

\begin{proposition}\label{coefabbc}
	Let $v$ be a highest weight vector of weight $(N,k,0)$, $N\geq1$
	and set $\lambda$ by
	\begin{equation*}
		A_{\alpha+\beta}B_{\alpha+\beta}v=\lambda v.
	\end{equation*}
	Then
	\begin{equation*}\label{abj}
		A_{\alpha+\beta}B_{\alpha+\beta}(A_{\beta+\gamma}^jv)=
		\lambda\frac{(j+1)(N-j)}N(A_{\beta+\gamma}^jv).
	\end{equation*}
	for $j\in\{0,\ldots,N-1\}$. Similarly, if
	\begin{equation*}
		B_{\beta+\gamma}A_{\beta+\gamma}v=\mu v,
	\end{equation*}
	then
	\begin{equation}\label{baj}
		B_{\beta+\gamma}A_{\beta+\gamma}(B_{\alpha+\beta}^jv)=
		\mu\frac{(j+1)(N-j)}N(B_{\alpha+\beta}^jv).
	\end{equation}
	for $j\in\{0,\ldots,N-1\}$
\end{proposition}

\begin{proof}
	It suffices to prove \eqref{abj}, as \eqref{baj} follows by a similar
	argument. The proof is by induction on $j$. For $j=0$, the statement becomes
	\begin{equation*}
		A_{\alpha+\beta}B_{\alpha+\beta}v=\lambda\frac NNv.
	\end{equation*}
	Assume that the statement holds for $j$, and let us prove it for $j+1$.
	Since (by Proposition \ref{commab})
	\begin{equation*}
		[A_{\alpha+\beta}B_{\alpha+\beta},A_{\beta+\gamma}]=
		Q(H_\alpha-H_\gamma)B_{\alpha+\beta}
	\end{equation*}
	it follows
	\begin{equation*}
		(A_{\alpha+\beta}B_{\alpha+\beta})A_{\beta+\gamma}=
		A_{\beta+\gamma}(A_{\alpha+\beta}B_{\alpha+\beta})
		+Q(H_\alpha-H_\gamma)B_{\alpha+\beta}.
	\end{equation*}
	Then we apply this relation to $A_{\beta+\gamma}^jv$,
	\begin{equation}\label{abind}
		(A_{\alpha+\beta}B_{\alpha+\beta})A_{\beta+\gamma}^{j+1}v=
		A_{\beta+\gamma}(A_{\alpha+\beta}B_{\alpha+\beta})A_{\beta+\gamma}^jv
		+Q(H_\alpha-H_\gamma)B_{\alpha+\beta}A_{\beta+\gamma}^jv.
	\end{equation}
	The left-hand side is the expression we need to compute.
	The first term on the right-hand side is obtained by applying the induction
	hypothesis. It remains to calculate the last term on the right-hand side,
	$Q(H_\alpha-H_\gamma)B_{\alpha+\beta}A_{\beta+\gamma}^jv$.
	The action of $B_{\alpha+\beta}$ on \eqref{npmbasis} (for $v=Qu$ and $j+1$)
	produces
	\begin{equation}\label{bq}
		B_{\alpha+\beta}A_{\alpha+\beta}A_{\beta+\gamma}^{j+1}u=
		B_{\alpha+\beta}(N-j)(j+1)A_{\beta+\gamma}^jQu.
	\end{equation}
	Observe that the equation
	\begin{equation*}
		B_{\alpha+\beta}A_{\alpha+\beta}A_{\beta+\gamma}^{j+1}u=
		\mu A_{\beta+\gamma}^{j+1}u,
	\end{equation*}
	is equivalent to
	\begin{equation*}
		A_{\alpha+\beta}B_{\alpha+\beta}A_{\beta+\gamma}^jv=
		\mu A_{\beta+\gamma}^jv
	\end{equation*}
	where the same value of $\mu$ appears in both equations.
	By induction hypothesis,
	\begin{equation*}
		\mu=\lambda\frac{(j+1)(N-j)}N.
	\end{equation*}
	and \eqref{bq} transforms to
	\begin{equation*}
		B_{\alpha+\beta}A_{\beta+\gamma}^jv=
		\frac{\lambda\frac{(j+1)(N-j)}N}{(N-j)(j+1)}A_{\beta+\gamma}^{j+1}u=
		\frac\lambda NA_{\beta+\gamma}^{j+1}u.
	\end{equation*}
	Since
	$(H_\alpha-H_\gamma)A_{\beta+\gamma}^{j+1}u=(N-2j-2)A_{\beta+\gamma}^{j+1}u$
	and $Qu=v$
	\begin{equation*}
		Q(H_\alpha-H_\gamma)B_{\alpha+\beta}A_{\beta+\gamma}^jv=
		\frac\lambda N(N-2j-2)A_{\beta+\gamma}^{j+1}v.
	\end{equation*}
	It remains to use this formula in \eqref{abind},
	\begin{equation*}
		(A_{\alpha+\beta}B_{\alpha+\beta})A_{\beta+\gamma}^{j+1}v=
		\left(\frac\lambda N(j+1)(N-j)+\frac\lambda N(N-2j-2)
		\right)A_{\beta+\gamma}^{j+1}v=
	\end{equation*}
	\begin{equation*}
		=\frac\lambda N\left((j+2)N-j^2-j-2j-2\right)A_{\beta+\gamma}^{j+1}v=
		\frac\lambda N\left((j+2)(N-(j+1))\right)A_{\beta+\gamma}^{j+1}v.
	\end{equation*}
	This completes the induction step and hence the proof (of \eqref{abj}).
\end{proof}

\section{\texorpdfstring{Examples and coefficient patterns of unitary 
		$(\g,K)$-modules}
		{Examples and coefficient patterns of unitary (g,K)-modules}}

Let
\begin{equation*}
	N=\min_{n,m} \{n+m\;|\: (n,k,m)\in\operatorname{Supp}_K(V)
		\text{ for some }k \}.
\end{equation*}
where $\operatorname{Supp}_K(V)$ is the set of weights of module $V$.
That $N$ is already mentioned (and defined) in \eqref{bign}.
There are two situations: $N=0$ and $N>0$.

\subsection{$N=0$}

There are two cases: $k=4l$ and $k=2+4l$ for some $l\in\Z$. Let us consider
the first case. Let $v_0$ be the highest weight vector of weight $(0,0,0)$.
Let us assume that there exists $v_l\neq0,\;\forall l\in\Z$ and
$v_{l+1}=Qv_l,\;\forall l\in\Z$. Recall that
\begin{equation*}
	H_\beta v_l=2l v_l.
\end{equation*}
Set
\begin{equation}\label{xdef}
	A_\beta B_\beta v_0=xv_0=x_0v_0
\end{equation}
for some $x\in\C$. By abuse of notation, sometimes, we will write
\begin{equation*}
	A_\beta B_\beta=x=x_0.
\end{equation*}
Since only real negative values of $x$ can give rise to unitary modules, we
assume $x\in\mathbb R$ and $x<0$. Further restrictions on $x$, when needed, will
become clear from the coefficients in the examples below.

Since $v_0$ has weight $(0,0,0)$, $H_\beta v_0=0$, $Qv_0$ has weight $(0,4,0)$
and $H_\beta Qv_0=2$. By \eqref{abu},
\begin{equation*}
	A_\beta B_\beta Qv_0=QA_\beta B_\beta v_0+
	Q(H_\alpha+H_\gamma+2)(1-H_\beta)v_0=Qxv_0+2Qv_0=(x+2)Qv_0.
\end{equation*}
Hence, the operator $A_\beta B_\beta$ acts on $Qv_0$ by multiplication by
$x_1=x+2$. Let us carry out one more calculation,
\begin{equation*}
	A_\beta B_\beta Q^2v_0=
	QA_\beta B_\beta Qv_0+Q(H_\alpha+H_\gamma+2)(1-H_\beta)Qv_0=
	Q(x+2)Qv_0-2Q^2v_0=xQ^2v_0.
\end{equation*}
Hence, the operator $A_\beta B_\beta$ acts on $Q^2v_0$ by multiplication by
$x_2=x$. We can continue, but it is easy to guess the formula. Let $v_l$
be the highest weight vector of weight $(0,4l,0)$. Set $x_l$ by
\begin{equation}\label{ab}
	A_\beta B_\beta v_l=x_lv_l.
\end{equation}
We claim that
\begin{equation}\label{xldef}
	x_l=x-2l(l-2),\quad l\in\Z.
\end{equation}
The proof goes by induction. The base of induction is clear: $x_0=x$.
Now, let us assume that the statement is proved for some $l\geq0$. Let us
notice that $n=m=0$ and $A_\beta v_l=0,\;\forall l\in\Z$. Now
\begin{align*}
	A_\beta B_\beta Qv_l&=QA_\beta B_\beta v_l+2Q(1-2l)v_l=\\
	&=(x-2l(l-2)+2-4l)Qv_l=(x-2l^2+2)Qv_l=\\
	&=(x-2(l+1)(l-1))Qv_l
\end{align*}
and it completes the inductive step for positive $l$. If $l$ is negative,
the statement also holds. The above calculation shows that the statement is
true for $l=-1$ (if it is not true for $l=-1$ then it is not true for $l=0$)
and then for $l=-2$ and so on. Also, we could use formulas for
$A_\beta B_\beta R$ that we have mentioned in Remark \ref{othcom}. Let us
recall that $QRv$ is some multiple of $v$, but the value of the composition
$A_\beta B_\beta$ remains unchanged.

Let us calculate expressions $B_{\alpha+\beta+\gamma}A_{\alpha+\beta+\gamma}$.
By \eqref{nmnula},
\begin{equation}\label{ba}
	B_{\alpha+\beta+\gamma}A_{\alpha+\beta+\gamma}=A_\beta B_\beta-2k=
	x-2l(l-2)-8l=x-2l^2+4l-8l=x-2l(l+2)
\end{equation}
Define
\begin{equation*}
	y_l=x-2l(l+2).
\end{equation*}
Obviously, $y_l=x_{l+2}$. However, sometimes we will use $y_l$.
It remains to check that \eqref{bad} is satisfied. Let $v_l$ be a highest
weight vector of weight $(0,k,0)=(0,4l,0)$.
Then the left-hand side of \eqref{bad} by \eqref{ba} is equal to
\begin{equation*}
	B_{\alpha+\beta+\gamma}A_{\alpha+\beta+\gamma}Rv_l=(x-2(l-1)(l+1))Rv_l.
\end{equation*}
The right-hand side of \eqref{bad} by \eqref{ba} is equal to
\begin{equation*}
	R(x-2l(l+2))v_l+2\cdot(2l+1)Rv_l=(x-2l^2-4l+4l+2)Rv_l=(x-2l^2+2)v_l.
\end{equation*}
It shows that \eqref{bad} is satisfied.

Let us summarize our results. For any unitary $(\g,K)$-module which contains
weight vectors $(0,4l,0)$ for $l\in\mathbb Z$ (and possibly some other
weights), let $x$ be a negative real number defined by \eqref{xdef}. Then the
coefficients $A_\beta B_\beta$ at weights $(0,4l,0)$ are given by
\eqref{xldef}, and the coefficients
$B_{\alpha+\beta+\gamma}A_{\alpha+\beta+\gamma}$ at weights $(0,4l,0)$ are
given by \eqref{ba}. All other coefficients are equal to $0$.

A few comments are worth mentioning. The calculations are similar when the
$(\g,K)$-module contains weights $(0,4l+2,0)$ instead of $(0,4l,0)$.
The Langlands classification generically involves two real parameters. We
will obtain another real parameter in the next section. Finally, a natural
question arises: can one construct a unitary $(\g,K)$-module for any choice
of $x$? There is strong evidence that the answer is positive, and we are
currently working on a rigorous proof.

\subsection{\texorpdfstring{$N>0$}{N greater than 0}}

There are two cases: In the first case, the set of weights
(more precisely, the set of weights for which there exist nonzero
highest weight vectors) contains weights of form $(0,k,N)$ (for some $k$) and
$(N,k,0)$. The set of weights of the form $(n,k,m)$ satisfying $n+m=N$ will be
determined later. However, we have already discussed coefficients in
Proposition \ref{coefabbc}. It should be noted that Proposition \ref{coefabbc}
requires the existence of weights $(N,k,0)$ and $(0,k,N)$.
In the second case, the set of weights has the form $\{(n,0,n),(n+1,-2,n+1),
(n+1,2,n+1),(n+2,-4,n+2),(n+2,0,n+2),(n+2,4,n+2),\ldots\}$ and all weights
have the multiplicity one. The set of weights
is significantly smaller and it will correspond to ladder representations.
At first glance, this may seem somewhat surprising. However, the explanation is
very simple: there are no weights of the form $(n,k,m)$ such that $n>m$,
$m\neq0$. This follows from Proposition \ref{system}.

Hence, let us analyze the set of nonzero highest weight vectors of weight
$(n,k,m)$ satisfying $n+m=N$. We aim to obtain the first or the second
case. Let us assume that there are weights such that $n\geq m$.
The case $n\leq m$ is treated similarly.
Let $v$ be the highest weight vector such that $n-m$ is maximal.
It means that that there are no vectors of weight $(n+1,k\pm2,m-1)$.
In particular, this implies that $B_{\alpha+\beta}A_{\alpha+\beta}=0$
and $A_{\beta+\gamma}B_{\beta+\gamma}=0$. Since $B_\beta A_\beta=0$ and
$A_{\alpha+\beta+\gamma}B_{\alpha+\beta+\gamma}=0$, the system given in
Proposition \ref{system} transforms to a system of 4 equations in 4 unknowns,
\begin{flushleft}
	$\ds \frac1{(n+2)(m+2)}A_\beta B_\beta
	-\frac1{(n+1)(m+1)(m+2)}B_{\beta+\gamma}A_{\beta+\gamma}-$
\end{flushleft}
\begin{flushright}
	$\ds -\frac1{(n+1)(n+2)(m+1)(m+2)}
	B_{\alpha+\beta+\gamma}A_{\alpha+\beta+\gamma}=\frac12(k-n-m)$
\end{flushright}
\begin{flushleft}
	$\ds \frac1{(n+1)(m+2)}A_{\alpha+\beta}B_{\alpha+\beta}
	+\frac1{(n+1)(n+2)(m+2)}A_\beta B_\beta-$
\end{flushleft}
\begin{flushright}
	$\ds -\frac1{(n+2)(m+1)(m+2)}
	B_{\alpha+\beta+\gamma}A_{\alpha+\beta+\gamma}=\frac12(k+n-m)$
\end{flushright}
\begin{flushleft}
	$\ds -\frac1{(n+1)(m+2)}B_{\beta+\gamma}A_{\beta+\gamma}
	+\frac1{(n+2)(m+1)(m+2)}A_\beta B_\beta-$
\end{flushleft}
\begin{flushright}
	$\ds -\frac1{(n+1)(n+2)(m+2)}
	B_{\alpha+\beta+\gamma}A_{\alpha+\beta+\gamma}=\frac12(k-n+m)$
\end{flushright}
\begin{flushleft}
	$\ds -\frac1{(n+2)(m+2)}B_{\alpha+\beta+\gamma}A_{\alpha+\beta+\gamma}
	+\frac1{(n+1)(m+1)(m+2)}A_{\alpha+\beta}B_{\alpha+\beta}+$
\end{flushleft}
\begin{flushright}
	$\ds +\frac1{(n+1)(n+2)(m+1)(m+2)}A_\beta B_\beta=\frac12(k+n+m)$
\end{flushright}
If $m=0$, then we have seen that the system reduces to system of 2 equations in
4 unknowns (Remark \ref{4to2}). This situation is discussed in
Proposition \ref{coefabbc} and it is the first case mentioned above.
Hence, let us assume that $m>0$.
Observe that $A_\beta B_\beta$ (with a positive sign) and
$B_{\alpha+\beta+\gamma}A_{\alpha+\beta+\gamma}$ (with a negative sign) occur
in all equations. Also, $B_{\beta+\gamma}A_{\beta+\gamma}$ (with a positive
sign) occur in the first and the third equation and
$A_{\alpha+\beta}B_{\alpha+\beta}$ (with a negative sign) occur in the second
and the fourth equation. It is easy to solve this system. Namely, one should
multiply the first equation by $m+1$ and subtract the third equation from that
product. It produces
\begin{equation*}
	\frac{m(m+2)}{(n+2)(m+1)(m+2)}A_\beta B_\beta=\frac12(k-n-m-2)m
\end{equation*}
or ($m>0$)
\begin{equation*}
	A_\beta B_\beta=\frac12(k-n-m-2)(n+2)(m+1).
\end{equation*}
Similarly, the second and the fourth equation produce
\begin{equation*}
	B_{\alpha+\beta+\gamma}A_{\alpha+\beta+\gamma}=\frac12(-k-n-m-2)(n+2)(m+1).
\end{equation*}
Now, it is easy to calculate remaining two coefficients,
\begin{equation*}
	B_{\beta+\gamma}A_{\beta+\gamma}=\frac12(-k+n-m)n(m+1)
\end{equation*}
and
\begin{equation*}
	A_{\alpha+\beta}B_{\alpha+\beta}=\frac12(k+n-m)n(m+1).
\end{equation*}
Hence
\begin{equation*}
	B_{\beta+\gamma}A_{\beta+\gamma}+A_{\alpha+\beta}B_{\alpha+\beta}=
	(n-m)n(m+1).
\end{equation*}
Now, we restrict our attention to unitary $(\g,K)$ modules. By \eqref{dneg},
coefficients $B_{\beta+\gamma}A_{\beta+\gamma}$ and
$A_{\alpha+\beta}B_{\alpha+\beta}$ must be non-positive real numbers.
It produces $n=m$ and $k=0$. Let us recall that $n+m=N$.

Let us consider the set of weights satisfying $n+m=N+2$. For weights other
than $(n+1,\pm2,n+1)$, the coefficients $B_\beta A_\beta$ and
$A_{\alpha+\beta+\gamma}B_{\alpha+\beta+\gamma}$ are both equal to $0$,
and we can repeat the above argument. Since $n+m=N+2$, Proposition~\ref{parity}
implies that $k\neq0$. It follows that the only possible weights are of the
form $(n+1,\pm2,n+1)$.

Continuing in the same way, we obtain weights of the form
\[ (n+2r,-2r+4s,n+2r), \qquad r,s\in\mathbb N\cup\{0\},\quad 0\leq s\leq r, \]
which gives the second case described above.

\section{\texorpdfstring{Conditional unitary coefficient patterns}
	{Conditional unitary coefficient patterns}}

\subsection{\texorpdfstring{$N=0$}{N=0}}\label{nis0}

We continue with the case where the $(\g,K)$ module $V$ contains the
weights $(0,4l,0)$ for $l\in\Z$. The case when $V$ contains weights $(0,4l+2,0)$
is analogous. We restrict our attention to irreducible unitary $(\g,K)$ modules.
We assume that for every weight $(0,4l,0)$ there exists a nonzero vector $v_l$.
We normalize the vectors $v_l$ by requiring that
\begin{equation*}
	\|v_l\|=1,\qquad \forall\, l\in\Z.
\end{equation*}
Since the norms of the vectors $v_l$ may be chosen arbitrarily, this is merely
a convenient normalization. The norms of all other vectors are then determined
by the coefficients.

Let us fix some $l$ and consider the vectors: $p=B_\beta v_l$ and
$q=A_{\alpha+\beta+\gamma}v_{l-1}$. These vector have the same weight, namely
$(1,4l-2,1)$. Let us calculate their norms. First, we compute $\|p\|^2=(p,p)$,
\begin{equation*}
	(p,p)=(B_\beta v_l,B_\beta v_l)=(B_\beta^*B_\beta v_l,v_l)
	\stackrel{\eqref{bstar}}{=}
	(-A_\beta\frac1{(H_\alpha+1)(H_\gamma+1)}B_\beta v_l,v_l)=
	-\frac{x_l}4(v_l,v_l)=-\frac{x_l}4.
\end{equation*}
Similarly, by \eqref{astar},
\begin{equation*}
	(q,q)=(A_{\alpha+\beta+\gamma}v_{l-1},A_{\alpha+\beta+\gamma}v_{l-1})=
	-\frac{y_{l-1}}4.
\end{equation*}

However, the vector $p$ and $q$ may be collinear.
Define
\begin{equation}\label{qlambda}
	Qv_{l-1}=\lambda v_l
\end{equation}
and
\begin{equation}\label{rmu}
	Rv_l=\mu v_{l-1}
\end{equation}
Now,
\begin{equation*}
	(p,q)=(B_\beta v_l,A_{\alpha+\beta+\gamma}v_{l-1})
	\stackrel{\eqref{bstar}}{=}
	-\frac14(v_l,A_\beta A_{\alpha+\beta+\gamma}v_{l-1})
	\stackrel{\eqref{comma}}{=}
	\frac12(v_l,Qv_{l-1})
	\stackrel{\eqref{qlambda}}{=}
	\frac12\overline{\lambda}
\end{equation*}
and
\begin{equation*}
	(p,q)=(B_\beta v_l,A_{\alpha+\beta+\gamma}v_{l-1})
	\stackrel{\eqref{astar}}{=}
	-\frac14(B_{\alpha+\beta+\gamma}B_\beta v_l,v_{l-1})
	\stackrel{\eqref{commb}}{=}
	\frac12(Rv_l,v_{l-1})
	\stackrel{\eqref{rmu}}{=}
	\frac12\mu.
\end{equation*}
This shows that $\overline{\lambda}=\mu$. By \eqref{qlambda} and \eqref{rmu},
\begin{equation*}
	QRv_l=|\lambda|^2v_l\qquad\mbox{and}\qquad RQv_{l-1}=|\lambda|^2v_{l-1}.
\end{equation*}
Also,
\begin{equation*}
	|\lambda|^2=|\lambda|^2(v_l,v_l)=(QRv_l,v_l)\stackrel{\eqref{qrstar}}{=}
	(Rv_l,Rv_l).
\end{equation*}
Let us denote $|\lambda|^2=|\mu|^2\in\R$ by $z_l$,
\begin{equation*}
	z_l=|\lambda|^2=|\mu|^2\geq0.
\end{equation*}
We can proceed with our analysis.
If $\lambda=0$ then we get reducibility which will not be considered here.
Hence, we assume $z_l>0$ and
one takes $w_{l-1}=\frac{\sqrt{z_l}}{\lambda}v_{l-1}$ instead of
$v_{l-1}$. It is easy to check that $(w_{l-1},w_{l-1})=1$,
\eqref{qlambda} transforms to
\begin{equation}\label{qlambdanew}
	Qw_{l-1}=\frac{\sqrt{z_l}}{\lambda}\lambda v_l=\sqrt{z_l}v_l
\end{equation}
and
\begin{equation}\label{rmunew}
	Rv_l=\sqrt{z_l}w_{l-1}.
\end{equation}
We can continue and choose a basis such that \eqref{qlambdanew} and
\eqref{rmunew} are satisfied.

Let us now, investigate $z_l$ further.
By the Cauchy-Schwarz inequality, we obtain
\begin{equation*}
	|(p,q)|^2\leq\|p\|^2\|q\|^2
\end{equation*}
and therefore
\begin{equation*}
	\frac14z_l\leq\frac1{16}x_ly_{l-1}
\end{equation*}
or
\begin{equation*}
	z_l\leq\frac14x_ly_{l-1}.
\end{equation*}

We begin with the case
\begin{equation*}
	z_l=\frac14x_ly_{l-1}.
\end{equation*}
We have the equality in the Cauchy-Schwarz inequality. Hence, vectors $p$ and
$q$ are collinear. In Theorem~\ref{qrel}, we established a relation involving
commutators. In our situation ($n=0$ and $m=0$), \eqref{qrrel} reduces to
\begin{equation*}
	QR-RQ=\frac12\left((H_\beta-1)y_l+(H_\beta+1)x_l\right).
\end{equation*}
It enables us to compute $z_{l+1}=RQ$,
\begin{equation*}
	z_{l+1}=RQ=QR-\frac12\left((H_\beta-1)y_l+(H_\beta+1)x_l\right)=
\end{equation*}
\begin{equation*}
	=\frac14x_ly_{l-1}
	-\frac12\left((2l-1)(x-2l(l+2))+(2l+1)(x-2l(l-2))\right)=
\end{equation*}
\begin{equation*}
	=\frac14x_ly_{l-1}
	-\frac12\left(4lx-2l((2l-1)(l+2)+(2l+1)(l-2))\right)=
\end{equation*}
\begin{equation*}
	=\frac14x_ly_{l-1}-\frac12\left(4lx-2l(4l^2-4)\right)
	=\frac14x_ly_{l-1}-2ly_{l-1}
	=\frac14y_{l-1}\left(x_l-8l\right)
	=\frac14y_{l-1}x_{l+2}=\frac14x_{l+1}y_l.
\end{equation*}
Hence,
\begin{equation*}
	z_{l+1}=\frac14x_{l+1}y_l
\end{equation*}
for our fixed $l$ and, by induction, if follows
\begin{equation*}
	z_{l+1}=\frac14x_{l+1}y_l.
\end{equation*}
for all $l\in\Z$.
Hence, vectors $A_{\alpha+\beta+\gamma}v_{l-1}$ and $B_\beta v_l$ are collinear
for all $l\in\Z$.
We can continue and compute vectors for weights $(2,k,2)$.
This is the subject of our current work.
It turns out that all such weights have multiplicity $1$.

We now consider the case
\begin{equation*}
	z_l<\frac14x_ly_{l-1}.
\end{equation*}
Again, let us fix one $l$. Define $D$ by
\begin{equation*}
	z_l=\frac14\left(x_ly_{l-1}-D\right)
\end{equation*}
for this fixed $l$. Let us use \eqref{qrrel} again. Since $z_l-z_{l+1}$ is
determined, if one (fixed) $z_l$ decreases by $\frac D4$, then all $z_l$
decrease by the same amount Hence, we define
\begin{equation*}
	z_l=\frac14\left(x_ly_{l-1}-D\right)
\end{equation*}
for all $l\in\Z$. It remains to determine the range of $D$. Since $z_l\geq0$,
\begin{equation}\label{ddef}
	0\leq D\leq\min_l\{x_ly_{l-1}\}.
\end{equation}
It is easy to determine $\min_l{x_ly_{l-1}}$ since $x_l$ and
$y_l=x_{l+2}$ are quadratic functions. If
\[ D=\min_l{x_ly_{l-1}}, \]
then at least one of the coefficients $z_l$ vanishes, and the corresponding
$Q/R$ transition disappears. This is precisely the boundary case associated
with reducibility; we do not analyze it further here.

Thus, the two real parameters $x$ and $D$, with $D$ satisfying
\eqref{ddef}, parametrize the coefficient systems obtained above. These
systems arise from the unitarity conditions and are intended to describe
irreducible unitary $(\g,K)$-modules. We do not claim here that every
admissible pair $(x,D)$ gives rise to such a module; the existence and
irreducibility questions require separate consideration. At the boundary
$D=\min_l{x_ly_{l-1}}$, one of the coefficients $z_l$ vanishes, leading
to the reducible case discussed above. In the other boundary case, $D=0$,
the resulting coefficient system has multiplicity one. A similar discussion
applies to the parity class $k=4l+2$.

The known results on the unitary dual of $SU(2,2)$ provide strong evidence that
these coefficient systems correspond to unitary $(\g,K)$-modules. A complete
construction from the coefficients obtained above will be considered separately.

\subsection{\texorpdfstring{$N>0$ and $n>m$}
	{N greater than 0 and n greater than m}}

In the preceding case $N=0$, the coefficient system is determined by two real
parameters. We shall see that the same is true in the present case: the
corresponding coefficient system is again determined by two real parameters.

If $N$ even, there are two parity classes: one class contains the weight
$(N,0,0)$ and the other class contains the weight $(N,2,0)$.
If $N$ is odd, there are again two parity classes: one class contains the weight
$(N,1,0)$ and the other contains the weight $(N,3,0)$.
Let us consider the first case.

Let us fix the weight $(N,0,0)$ and denote the coefficient
$A_{\alpha+\beta}B_{\alpha+\beta}$ by $x$ and the coefficient
$B_{\beta+\gamma}A_{\beta+\gamma}$ by $y$. By Proposition~\ref{abbcmnul},
within a component on which the stated transitions remain nonzero,
we can compute the coefficients $A_{\alpha+\beta}B_{\alpha+\beta}$ and
$B_{\beta+\gamma}A_{\beta+\gamma}$ at all weights of the form $(N,4l,0)$.
For example, for the weight $(N,4,0)$,
\begin{equation*}
	A_{\alpha+\beta}B_{\alpha+\beta}=x-N\frac N2.
\end{equation*}
By Proposition \ref{coefabbc}, one can calculate coefficients
$A_{\alpha+\beta}B_{\alpha+\beta}$ and $B_{\beta+\gamma}A_{\beta+\gamma}$
at all weights of the form $(n,k,m)$ satisfying $n+m=N$. For example,
for the weight $(N-1,2,1)$,
\begin{equation*}
	A_{\alpha+\beta}B_{\alpha+\beta}=x\frac{2(N-1)}N
\end{equation*}
Finally, by Proposition~\ref{system}, one can compute coefficients
$B_{\alpha+\beta+\gamma}A_{\alpha+\beta+\gamma}$ and $A_\beta B_\beta$.
It is clear that Proposition~\ref{system} produces
$B_{\alpha+\beta+\gamma}A_{\alpha+\beta+\gamma}$ and $A_\beta B_\beta$
at weights $(N,k,0)$ and $(0,k,N)$.
Namely, we have a system of 2 equations (by Remark~\ref{4to2}) in 2 unknowns.
However, we have four equations in two unknowns, so the system is
overdetermined and imposes additional compatibility conditions on $x$ and
$y$. Nevertheless, a direct calculation of all the coefficients shows that
these compatibility conditions are satisfied.

It is clear that $x<0$ and $y<0$ since $\alpha+\beta$ and $\beta+\gamma$ are
noncompact roots. Also, one can use Proposition~\ref{star} to prove this
statement. It is easy to give a precise restriction on $x$ and $y$. We will not
do it now.

\subsection{\texorpdfstring{$N>0$ and $n=m$}{N greater than 0 and n=m}}

We have already seen that the set of $K$ types (more precisely, the set of
weights of highest weight vectors) has the form
\begin{equation}\label{weights}
	\{(n+p+q,2p-2q,n+p+q)\:|\:n\in\N,p,q\in\N\cup\{0\}\}
\end{equation}
and multiplicities are always 1. Namely, we start with the weight $(n,0,n)$
and then apply operators $A_{\alpha+\beta+\gamma}$ ($p$ times) and
$B_\beta$ ($q$ times). It remains to calculate our coefficients.
Obviously, $A_{\alpha+\beta}B_{\alpha+\beta}=0$,
$B_{\alpha+\beta}A_{\alpha+\beta}=0$, $A_{\beta+\gamma}B_{\beta+\gamma}=0$ and
$B_{\beta+\gamma}A_{\beta+\gamma}=0$.
We use the system of 4 equations given by Proposition \ref{system}.
Since $n=m$, $B_{\beta+\gamma}A_{\beta+\gamma}=0$
$A_{\beta+\gamma}B_{\beta+\gamma}=0$,
$A_{\alpha+\beta}B_{\alpha+\beta}=0$ and
$B_{\alpha+\beta}A_{\alpha+\beta}=0$, the system can be simplified.
Also, we will write
$n+p+q$ instead of $n$ and $m$ and $2p-2q$ instead of $k$,
\begin{flushleft}
	$\ds \frac1{(n+p+q+2)^2}A_\beta B_\beta-
	\frac1{(n+p+q+1)^2}B_\beta A_\beta-$
\end{flushleft}
\begin{flushright}
	$\ds -\frac1{(n+p+q+1)^2(n+p+q+2)^2}
	B_{\alpha+\beta+\gamma}A_{\alpha+\beta+\gamma}=-n-2q$
\end{flushright}
\begin{equation*}
	\frac1{(n+p+q+1)(n+p+q+2)^2}A_\beta B_\beta-\frac1{(n+p+q+1)(n+p+q+2)^2}
	B_{\alpha+\beta+\gamma}A_{\alpha+\beta+\gamma}=p-q
\end{equation*}
\begin{equation*}
	\frac1{(n+p+q+1)(n+p+q+2)^2}A_\beta B_\beta-\frac1{(n+p+q+1)(n+p+q+2)^2}
	B_{\alpha+\beta+\gamma}A_{\alpha+\beta+\gamma}=p-q
\end{equation*}
\begin{flushleft}
	$\ds \frac1{(n+p+q+1)^2}A_{\alpha+\beta+\gamma}B_{\alpha+\beta+\gamma}
	-\frac1{(n+p+q+2)^2}B_{\alpha+\beta+\gamma}A_{\alpha+\beta+\gamma}+$
\end{flushleft}
\begin{flushright}
	$\ds +\frac1{(n+p+q+1)^2(n+p+q+2)^2}A_\beta B_\beta=n+2p$
\end{flushright}
Observe that the second and the third equations coincide.
Let us solve this system for all $p$ and $q$.
We start with $p=0$ and $q$=0 and assume $B_\beta A_\beta=0$ and
$A_{\alpha+\beta+\gamma}B_{\alpha+\beta+\gamma}=0$. Now, we have a system
of 3 equations in 2 unknowns. The system has a solution which satisfies
all equations:
$B_{\alpha+\beta+\gamma}A_{\alpha+\beta+\gamma}(0,0)=
A_\beta B_\beta(0,0)=-(n+1)^2(n+2)$. We included $(0,0)$ to emphasize
that coefficients are computed for $p=0$ and $q=0$.
We use this solution to solve systems for $p+q=1$. For $p=0$ and $q=1$,
$A_\beta B_\beta(0,1)=-2(n+2)^2(n+3)$ and
$B_{\alpha+\beta+\gamma}A_{\alpha+\beta+\gamma}(0,1)=-(n+1)(n+2)(n+3)$.
For $p=1$ and $q=0$, $A_\beta B_\beta(1,0)=-(n+1)(n+2)(n+3)$ and
$B_{\alpha+\beta+\gamma}A_{\alpha+\beta+\gamma}(1,0)=-2(n+2)^2(n+3)$.

We can continue and calculate coefficients for $p+q=2$ and $p+q=3$. Based on
these computations, we conjecture the following formula for the coefficients,
\begin{equation}\label{abbpq}
	A_\beta B_\beta(p,q)=-(q+1)(n+q+1)(n+p+q+1)(n+p+q+2),
\end{equation}
\begin{equation}\label{baabcpq}
	B_{\alpha+\beta+\gamma}A_{\alpha+\beta+\gamma}(p,q)=
	-(p+1)(n+p+1)(n+p+q+1)(n+p+q+2),
\end{equation}
\begin{equation}\label{babpq}
	B_\beta A_\beta(p,q)=-q(n+q)(n+p+q)(n+p+q+1)
\end{equation}
and
\begin{equation}\label{ababcpq}
	A_{\alpha+\beta+\gamma}B_{\alpha+\beta+\gamma}(p,q)=-p(n+p)(n+p+q)(n+p+q+1).
\end{equation}
Note that $B_\beta A_\beta(p,q)=A_\beta B_\beta(p,q-1)$ and
$A_{\alpha+\beta+\gamma}B_{\alpha+\beta+\gamma}(p,q)=
B_{\alpha+\beta+\gamma}A_{\alpha+\beta+\gamma}(p-1,q)$.
For $p=0$ and $q=0$, $B_\beta A_\beta=0$ and
$A_{\alpha+\beta+\gamma}B_{\alpha+\beta+\gamma}=0$ as we assumed.

Let us show that the first equation is satisfied. Formulas \eqref{abbpq},
\eqref{babpq} and \eqref{baabcpq} yield
\begin{flushleft}
	$\ds -\frac1{n+p+q+2}(q+1)(n+q+1)(n+p+q+1)+\frac1{n+p+q+1}q(n+q)(n+p+q)+$
\end{flushleft}
\begin{flushright}
	$\ds +\frac1{(n+p+q+1)(n+p+q+2)}(p+1)(n+p+1)=-n-2q$
\end{flushright}
which can be rewritten as
\begin{flushleft}
	$\ds (q+1)(n+q+1)(n+p+q+1)^2-q(n+q)(n+p+q)(n+p+q+2)-$
\end{flushleft}
\begin{flushright}
	$\ds -(p+1)(n+p+1)=(n+2q)(n+p+q+1)(n+p+q+2)$
\end{flushright}
Although this expression appears rather complicated, it can be proved directly,
\begin{flushleft}
	$\ds (q(n+q)+(n+2q)+1)(n+p+q+1)^2-q(n+q)(n+p+q+1)^2+q(n+q)-$
\end{flushleft}
\begin{flushright}
	$\ds -(p+1)(n+p+1)=(n+2q)(n+p+q+1)^2+(n+2q)(n+p+q+1)$.
\end{flushright}
After simplification, we obtain
\begin{equation*}
	(n+p+q+1)^2+q(n+q)-(p+1)(n+p+1)=(n+2q)(n+p+q+1).
\end{equation*}
We can simplify this expression further and obtain
\begin{equation*}
	(p-q+1)(n+p+q+1)=(p+1)(n+p+1)-q(n+q).
\end{equation*}
It is straightforward to verify that the expression is valid.

The second (and third) equation is easier to verify. We start with
\begin{equation*}
	-(q+1)(n+q+1)+(p+1)(n+p+1)=(p-q)(n+p+q+2).
\end{equation*}
After simplification, we obtain
\begin{equation*}
	-(q+1)(q+1)+(p+1)(p+1)=(p-q)(p+q+2)
\end{equation*}
and the claim follows.

The proof of the last equation is analogous to that of the first and is omitted.
Hence, we have proved

\begin{proposition}
	Suppose that $V$ is in the second case described above, so that its weights
	are of the form \eqref{weights} and have multiplicity one. Then, for every
	$n\geq1$ and $p,q\geq0$, the coefficients are given by
	\eqref{abbpq}--\eqref{ababcpq}. These formulas satisfy the four scalar
	equations above and the two predecessor identities following those formulas,
	and all the resulting coefficients are nonpositive.
\end{proposition}

\section{Unitary dual of $SU(2,2)$}

Knapp and Speh determined the irreducible unitary representations of
$SU(2,2)$ in \cite{KnappSpeh1982}. This provides an opportunity to compare
their parametrization with our construction and to investigate the
relationship between the two approaches. Such a comparison requires a careful
analysis of all cases, particularly those involving reducibility. We do not
undertake a complete parameter-by-parameter comparison in the present paper,
but hope to return to this question in future work.

\bibliographystyle{alpha} 
\bibliography{b}

\end{document}